\documentclass[12pt]{article}
\usepackage{graphicx}

\usepackage{amssymb}
\usepackage{amsmath,amsthm}
\usepackage[english]{babel}
\usepackage[T1]{fontenc}
\usepackage[a4paper,top=2.5cm,bottom=3cm,left=2.5cm,right=2.5cm,marginparwidth=1.75cm]{geometry}

\usepackage{microtype}

\usepackage{enumitem} %

\usepackage{mathtools}
\usepackage{graphicx,adjustbox}
\usepackage[dvipsnames]{xcolor}
\usepackage{bbm}
\usepackage[colorlinks=true, allcolors=black]{hyperref}
\usepackage[normalem]{ulem}
\usepackage{cancel}
\usepackage{varwidth}  %
\usepackage{booktabs}

\usepackage{xparse}
\NewDocumentCommand{\grad}{e{_^}}{%
  \mathop{}\!%
  \nabla
  \IfValueT{#1}{_{\!#1}}%
  \IfValueT{#2}{^{#2}}%
}

\newcommand{\R}{\bbR}
\newcommand{\e}{\varepsilon}

\DeclareMathOperator{\tr}{tr}
\DeclareMathOperator*{\argmin}{arg\, min}

\DeclareMathOperator{\Lip}{Lip}

\def\longrightharpoonup{\relbar\joinrel\rightharpoonup}

\usepackage{xspace}

\usepackage{dsfont}

\newtheorem{theorem}{Theorem}[section]
\newtheorem{corollary}[theorem]{Corollary}
\newtheorem{lemma}[theorem]{Lemma}
\newtheorem{proposition}[theorem]{Proposition}
\newtheorem{assumption}[theorem]{Assumption}

\theoremstyle{definition}
\newtheorem{definition}[theorem]{Definition}
\usepackage{textcomp}
\newtheorem{remarkx}[theorem]{Remark}
\newtheorem{examplex}[theorem]{Example}
\newenvironment{remark}
{\pushQED{\qed}\remarkx} %
{\popQED\endremarkx}
\newenvironment{example}
{\pushQED{\qed}\examplex} %
{\popQED\endexamplex}

\DeclarePairedDelimiter{\abs}{\lvert}{\rvert}
\DeclarePairedDelimiter{\norm}{\lVert}{\rVert}
\DeclarePairedDelimiter{\bra}{(}{)}
\DeclarePairedDelimiter{\pra}{[}{]}
\DeclarePairedDelimiter{\set}{\{}{\}}
\DeclarePairedDelimiter{\skp}{\langle}{\rangle}
\DeclareMathAlphabet{\mathup}{OT1}{\familydefault}{m}{n}

\usepackage{ifthen}
\newlength{\leftstackrelawd}
\newlength{\leftstackrelbwd}
\def\leftstackrel#1#2{\settowidth{\leftstackrelawd}%
{${{}^{#1}}$}\settowidth{\leftstackrelbwd}{$#2$}%
\addtolength{\leftstackrelawd}{-\leftstackrelbwd}%
\leavevmode\ifthenelse{\lengthtest{\leftstackrelawd>0pt}}%
{\kern-.5\leftstackrelawd}{}\mathrel{\mathop{#2}\limits^{#1}}}

 \def\calE{{\mathcal E}} \def\calF{{\mathcal F}}

\def\calM{{\mathcal M}}  
\def\calP{{\mathcal P}}  
\def\calS{{\mathcal S}}  
  
\def\calY{{\mathcal Y}} 

\def\rmD{{\mathrm D}}

\def\bbD{{\mathbb D}}  
  
 \def\bbK{{\mathbb K}} 
 \def\bbN{{\mathbb N}} 
  \def\bbR{{\mathbb R}}
\def\bbS{{\mathbb S}} \def\bbT{{\mathbb T}} 
  
 \def\bbZ{{\mathbb Z}}

\newcommand{\dist}{\operatorname{dist}}
\newcommand{\proj}{\mathrm{proj}}

\newcommand{\divr}{\operatorname{div}}

\title{Nonlinear diffusion limit of non-local interactions \\on a sphere}
\author{Mark Peletier$^{1}$   \quad Anna Shalova$^{1, 2,}$\thanks{\href{mailto:a.shalova@uva.nl}{a.shalova@uva.nl}. The research was conducted while AS was at the Technical University of Eindhoven.}\vspace{0.3em} \\
\normalsize $^{1}$ Department of Mathematics and Computer Science,
\\
\normalsize Eindhoven University of Technology, \\
\normalsize $^{2}$ Korteweg-de Vries Institute for Mathematics,
\\
\normalsize University of Amsterdam} 

\begin{document}

\maketitle

\begin{abstract}
We study an aggregation PDE with competing attractive and repulsive forces on a sphere of arbitrary dimension. In particular, we consider the limit of strongly localized repulsion with a constant attraction term. We prove convergence of solutions of such a system to solutions of the aggregation-diffusion equation with a porous-medium-type diffusion term. The proof combines variational techniques with elements of harmonic analysis on a sphere. In particular, we characterize the square root of the convolution operator in terms of the spherical harmonics, which allows us to overcome difficulties arising due to the convolution on a sphere being non-commutative. The study is motivated by the toy model of transformers introduced by Geshkovski et al. \cite{geshkovski2024mathematical}; and we discuss the applicability of the results to this model.
\end{abstract}

\tableofcontents

\section{Introduction}

\subsection{Aggregation equations with and without diffusion}
In this paper we consider an aggregation equation on a sphere $\bbS^{n-1}$ in the presence of a both attractive and repulsive interactions. Concretely, we study measure-valued solutions $\rho_t: [0,T] \to \calP(\bbS^{n-1})$ of the equation
\begin{equation}
\tag{AE}
\label{eq:intro-main}
\partial_t\rho_t = \nabla \cdot (\rho_t \nabla W *\rho_t) + \nabla \cdot (\rho_t \nabla V_\varepsilon *\rho_t),
\end{equation}
where $\nabla \cdot$ and $\nabla$ are the spherical divergence and gradient and the symbol $*$ denotes spherical convolution, which is defined as
\[
(U*\mu)(x) := \int_{\bbS^{n-1}} U(x, y)d\mu(y).
\]

In the equation \ref{eq:intro-main}, the function $W\in C^2(\bbS^{n-1}\times \bbS^{n-1},\bbR)$ is a fixed interaction kernel and $(V_\varepsilon)_{\varepsilon >0}$ is a family of repulsive interaction kernels satisfying $V_\varepsilon\in C^1(\bbS^{n-1}\times \bbS^{n-1}, \bbR)$; we assume that both $W$ and $V_\varepsilon$ are rotationally symmetric, namely take the form $W(x, y) = W(\left<x, y\right>)$. 

In this work we consider the case of localized repulsion, corresponding to the limit in which the repulsive kernels $V_\varepsilon$ converge to a delta function as $\varepsilon \to 0$. We show that in this regime the solutions $(\rho^\e)_{\varepsilon>0}$ of the aggregation equation \eqref{eq:intro-main} converge to solutions of an aggregation-\emph{diffusion} equation with porous-medium-type nonlinear diffusion, 
\begin{equation}
\tag{ADE}
    \label{eq:intro-diffuse}
    \partial_t\rho_t = \nabla \cdot (\rho_t \nabla W *\rho_t) + \frac{1}{2}\Delta \rho_t^2,
\end{equation}
where $\Delta = \nabla \cdot \nabla$ is the Laplace-Beltrami operator. With a slight abuse of notation we write $\mu(x)$ for the density of a measure $\mu(dx)$ that is absolutely continuous with respect to the uniform probability measure $\sigma$ on $\bbS^{n-1}$, and therefore the expression $\rho_t^2$ should be read as the square of this  density of $\rho_t$.  Our proof combines the approach presented in~\cite{burger2023porous} (see also~\cite{matthes2009family}) with the spectral analysis of convolution on a sphere used in~\cite{shalova2024anoisy-transformer}. The spectral approach that we develop here for the Hilbert space $L^2(\bbS^{n-1})$ can be generalized to convolution operators on a larger class of compact manifolds, and we discuss this in more detail in the last section.

Note that depending on $W$, the corresponding interaction term can describe both attractive and repulsive interaction. We do not assume attractive or repulsive behaviour of $W$,  but we remark that the more interesting behavior appears when $W$ favors localized solutions, also called clusters. In this case the equation \eqref{eq:intro-main} can be interpreted as balancing counteracting long-range attractive and short-range repulsive forces. This is exactly the case in the main motivating example, the toy transformer model of~\cite{geshkovski2024mathematical}. We introduce this example in Section \ref{sec:relevance}, discuss the relevance of the global-attraction local-repulsion setting for the toy transformers, and outline the key challenges for applying our theoretical findings to actual transformers.

\paragraph{Heuristic explanation.} 
We now give a non-rigorous explanation why the equation~\eqref{eq:intro-main} should converge to~\eqref{eq:intro-gf} as $\e\to0$.
Equation~\eqref{eq:intro-main} has a gradient-flow structure in the space of probability measures $\calP(\bbS^{n-1})$ in the sense that it admits a representation of the form
\begin{equation}
\label{eq:intro-gf}
\partial_t\rho_t = \nabla \cdot \left(\rho_t\nabla \frac{\delta \calF_{\varepsilon}}{\delta \rho}(\rho_t)\right),
\end{equation}
where $\calF_\varepsilon: \calP(\bbS^{n-1}) \to \bbR$ is the energy functional %
defined as
\begin{align}
   \MoveEqLeft \calF_\varepsilon(\rho):= \frac{1}{2}\int_{\bbS^{n-1} \times \bbS^{n-1}} W(x, y)\rho(x)\rho(y)d\sigma(x)d\sigma(y) \notag \\
   &\qquad\qquad+\frac{1}{2}\int_{\bbS^{n-1} \times \bbS^{n-1}} V_\varepsilon(x, y)\rho(x)\rho(y)d\sigma(x)d\sigma(y) 
   \label{eq:intro-free-energy}
    \end{align}
and $\frac{\delta \calF_{\varepsilon}}{\delta \rho}$ is the variational derivative of $\calF_\varepsilon$. 
Evolution equations of this  form are known as Wasserstein gradient flows, see \eqref{eq:wasserstein} for the definition of the Wasserstein distance. %

Consider a solution $\rho_t$ of the aggregation equation \eqref{eq:intro-main} that admits a density $\rho_t = \frac{d\rho_t}{d\sigma}$ with respect to the uniform spherical measure~$\sigma$. Our assumptions on the repulsive kernel~$V_\e$ in Assumption~\ref{ass:repulsive} below imply that for every $x\in \bbS^{n-1}$ the measures $V_\varepsilon(x,y )\sigma(dy)$ converge to the measure  $\delta_x$. Therefore the free energy functional \eqref{eq:intro-free-energy} $\Gamma$-converges to the limit $\calF_0$,
\[
\calF_0(\rho) 
= \frac{1}{2}\int_{\bbS^{n-1} \times \bbS^{n-1}} W(x, y)\rho(x)\rho(y)d\sigma(x)d\sigma(y) 
+\frac{1}{2}\int_{\bbS^{n-1} \times \bbS^{n-1}}\rho^2(x)d\sigma(x).
\]
Calculating the first variation of the limiting free-energy functional we obtain $\frac{\delta \calF_{0}}{\delta \rho_t} = W*\rho_t + \rho_t$. Substituting this into the gradient flow equation \eqref{eq:intro-gf} yields
\[
\partial_t\rho_t = \nabla \cdot (\rho_t \nabla W *\rho_t) + \nabla \cdot(\rho_t \nabla \rho_t)= \nabla \cdot (\rho_t \nabla W *\rho_t) + \frac{1}{2}\Delta \rho_t^2.
\]
This heuristic calculation shows how the localized repulsive interactions converge to non-linear diffusion. We remark that the above argument is informal and we only provide it here for illustrative purposes. Note that similar results have been established in various settings on flat space \cite{oelschlager2001sequence, burger2023porous}, and we give a more detailed overview of the existing results in the next section.

\subsection{Related work}
\paragraph{Non-linear diffusion limit of non-local interactions.}
The convergence of localized repulsion to non-linear diffusion has been relatively widely studied in the Euclidean setting. One of the first results in this direction is the work of Oelschl\"ager \cite{oelschlager1990large}, in which the porous medium equation is recovered as the limiting dynamics of a system of deterministic interacting particles with localized repulsion. Later Philipowski and Figalli proved a similar convergence to nonlinear diffusion equations for a sequence of stochastic particle systems with vanishing noise~\cite{oelschlager2001sequence,philipowski2007interacting, figalli2008convergence}.  In \cite{carrillo2019blob}, the solutions of the porous medium equation are approximated by gradient-flow solutions of the regularized energy functional. We remark that the regularization arising in the blob method is exactly the convolution with a strongly localized kernel. An inhomogeneous counterpart of the latter result has been recently introduced in \cite{craig2023blob}. Recently, the rate of convergence of the nonlocal-to-local limit has been established for a specific choice of mollifier in one dimension in~\cite{carrillo2025rate}.

This work relies on a different approximation of the solutions of the porous-medium equation developed in \cite{burger2023porous}, which makes use of the gradient flow structure of the underlying system. This approach has been recently extended to a larger class of non-linear diffusion equations in~\cite{carrillo2024nonlocal}. %

\paragraph{Aggregation equations with nonlinear diffusion} We refer the reader to \cite{carrillo2019aggregation} for an overview of results concerning aggregation-diffusion equations on $\bbR^d$ and only mention a few reference points. Existence and uniqueness of stationary solutions are studied in \cite{carrillo2019existence, delgadino2022uniqueness}, and  bifurcation branches are characterized in \cite{CarrilloGvalani2021}. Existence results for time-dependent solutions to aggregation-diffusion equations on flat spaces are established in various settings in \cite{BurgerCapassoMorale2007, bertozzi2010existence}. 

\paragraph{Aggregation PDEs on manifolds}
Aggregation-diffusion PDEs on manifolds is a topic of an active research. In particular, stationary solutions of aggregation equations with linear diffusion are studied in \cite{fetecau2023equilibria, fetecau2023ground, carrillo2024existence, shalova2024anoisy-transformer} and with nonlinear diffusion in \cite{carrillo2025global}. In \cite{fetecau2023equilibria} the authors study the stationary solutions of the aggregation PDE on Cartan-Hadamar (hyperbolic) manifolds. Existence and long-time behaviour of the time-dependent solutions of interaction models on manifolds of bounded curvature are characterized in \cite{fetecau2023long,fetecau2021intrinsic}, in both cases for initial data with support in a particular strict subset of the manifold.

\subsection{Main contributions}
The main contributions of this paper are
\begin{enumerate}
\item We prove convergence of solutions of \eqref{eq:intro-main} to solutions of \eqref{eq:intro-diffuse} on a sphere of arbitrary dimension.
\item We prove that for any well-behaved initial condition $\rho_0$ the solution $\rho_t$ of~\eqref{eq:intro-diffuse} admits a density for arbitrary $t\in \bbR_+$, and the density is bounded in $L^2$ on any interval $(0, t)$.
\item We relate the equation \eqref{eq:intro-main} to the toy transformer model introduced in \cite{geshkovski2024mathematical} and, based on the presented analytical results, give an interpretation of the role of the repulsive heads in transformer models.
\end{enumerate}
We remark that in \cite{burger2023porous}, convergence of the solutions of \eqref{eq:intro-main} to solutions of \eqref{eq:intro-diffuse} is shown under a structural assumption on the localized kernel: it is assumed that there exists a `convolution square root' of $V_\varepsilon$, namely a function~$\sqrt[*]{V_\varepsilon}$ satisfying $V_\varepsilon = \sqrt[*]{V_\varepsilon} *\sqrt[*]{V_\varepsilon}$. In this work we give a sufficient condition for the existence of $\sqrt[*]{V_\varepsilon}$ in terms of the spherical harmonics decomposition of~$V_\varepsilon$. %
The approach can also easily extended to the setting of the torus $\bbT^d$ and to other compact Riemannian manifolds and we discuss this in more detail in Section~\ref{sec:discussion}.

\subsection{Notation}

We write $\calP(\bbS^{n-1})$ for the set of probability measures. The `uniform' measure $\sigma\in\calP(\bbS^{n-1})$ is the normalized spherical measure (the $(n{-}1)$-dimensional Hausdorff measure on $\bbS^{n-1}$), or equivalently the normalized volume measure on the sphere equipped with the metric generated by the standard Euclidean product in $\bbR^{n}$. We write $\rho_n\stackrel w\to \rho$ for the weak convergence in $\calP(\bbS^{n-1})$, which is generated by duality with continuous functions on $\bbD^{n-1}$.

The Hilbert space $L^2(\bbS^{n-1})$ is the set of (equivalence classes of) square-integrable functions on $\bbS^{n-1}$ equipped with the scalar product
    \[
    \left<f, g\right> = 
    \int_{\bbS^{n-1}}f(x)g(x)d\sigma(x).
    \]

We also often consider elements  $\rho$ in the intersection $\calP(\bbS^{n-1})\cap L^2(\bbS^{n-1})$. In this case we implicitly assume that $\rho$ is absolutely continuous with respect to $\sigma$, with a density that we denote as $\rho(x)$; explicitly, we consider that $\rho(dx) = \rho(x) \sigma(dx)$. 
We also use the  notation $u_n\stackrel w \to u$ for weak convergence for elements of $L^2$ and other Hilbert spaces, which is defined as usual in duality with the same Hilbert space. 

We write $g$ for the Riemannian metric on $\bbS^{n-1}$. The operators $\nabla$,  $\nabla\cdot$, and $\Delta$  always indicate the spherical gradient and spherical divergence and the Laplace-Beltrami operator. The space $H^1(\bbS^{n-1})$ is the Sobolev space of weakly differentiable functions with squared norm
\[
\|u\|_{H^1(\bbS^{n-1})}^2 := \|u\|_{L^2(\bbS^{n-1})}^2 + \|\nabla u\|_{L^2(T\bbS^{n-1})}^2.
\]
We give more background on the differential geometry that we use in Appendix~\ref{sec:geometry}.

\paragraph*{Acknowledgements}
The authors are grateful to Rafael Bailo, Nicolas Boumal, Jasper Hoeksema and Jim Portegies for helpful discussions. The work was supported by the Dutch Research Council (NWO), in the framework of the program `Unraveling Neural Networks with Structure-Preserving Computing' (file number OCENW.GROOT.2019.044).

\section{Properties of the interaction kernels}
\label{sec:kernels}
In this section we introduce and explain the main assumptions on both attractive and repulsive interaction kernels. Since most of the properties of the kernels are formulated in terms of the spherical harmonics, we  give a short introduction to these and to the convolution operator on a sphere in Sections~\ref{ssec:harmonics} and~\ref{ssec:properties}. After that, we formulate the assumptions on the interaction kernels in see Section \ref{sec:repulsive}. Finally, we introduce the necessary estimates in Section \ref{ssec:estimates}.

\subsection{Spherical harmonics}
\label{ssec:harmonics}
    The orthonormal basis of $L^2(\bbS^{n-1})$ known as the `spherical harmonics basis' can be constructed as follows, see e.g. \cite[Chapter 1.5]{dai2013approximation} %
    Introduce the spherical coordinates $\theta_1,\dots,\theta_{n-1}$ on $\bbS^{n-1}$, such that for all $x\in \bbS^{n-1}$:
    \begin{align*}
        x_1 &= r\sin \theta_{n-1} \cdots \sin\theta_{2}\sin\theta_{1}, \\
        x_2 &= r\sin \theta_{n-1} \cdots  \sin\theta_{2}\cos\theta_{1}, \\
        x_3 &= r\sin \theta_{n-1} \cdots  \cos\theta_{2}, \\
        &\vdots\\
        x_n &= r\cos\theta_{n-1}.
    \end{align*}
The corresponding basis of spherical harmonics in the given spherical coordinates is given by: %
\begin{equation*}
\label{eq:harmonics}
Y_{l, k}(\theta) = e^{ik_{n-2}\theta_1}A_k^l \prod_{j=0}^{n-3} C_{k_j-k_{j+1}}^{\frac{n-j-2}{2}+k_{j+1}}(\cos \theta_{n-j-1}) (\sin \theta_{n-j-1})^{k_{j+1}},
\end{equation*}
where $l\in \bbN_0$, $k\in \bbK_l$ is a multi-index satisfying
\begin{equation*}%
\bbK_l := \set*{k = (k_0, k_1 ,\dots, k_{n-2})\in \bbN_0^{n-2}\times \bbZ: l\equiv k_0 \geq k_1 \geq \cdots \geq k_{n-3} \geq |k_{n-2}|\geq 0},
\end{equation*}
$A_k^l$ is a normalization constant and $C^\lambda_m$ is the Gegenbauer polynomial of degree $m$. 
\begin{definition}[Gegenbauer polynomials]
Gegenbauer polynomials are defined recursively to satisfy the following relation:
\begin{equation*}
(n+2)C_{n+2}^\lambda(t) = 2(\lambda+n+1)tC_{n+1}^\lambda(t) -(2\lambda+n)C_{n}^\lambda(t),
\end{equation*}
where the first two polynomials are given by $C_0^\lambda(t) = 1$ and $C_1^\lambda(t) = 2\lambda t$.
\end{definition}

Note that by definition spherical harmonics are smooth functions. Moreover, spherical harmonics are eigenfunctions of the Laplace-Beltrami operator with eigenvalues depending only on the index $l$:
\[
\lambda_l = -l(n-2+l),
\]
and
thus (by the Hilbert-Schmidt theorem) form an orthonormal basis on $L^2(\bbS^{n-1})$. In particular, if we define the projection operator  onto the $l$-th subspace by $\proj_l: L^2(\bbS^{n-1}) \to L^2(\bbS^{n-1})$ 
\[
\proj_l f:= \sum_{k\in \bbK_l} Y_{l, k} \left<f, Y_{l, k}\right>,
\]
then the following theorem holds.
\begin{theorem}[Fourier decomposition on $\bbS^{n-1}$ {\cite[Theorem 2.2.2]
{dai2013approximation}}]
\label{th:sph-decomposition}
Let $Y_{l, k}$  be the spherical harmonics defined by \eqref{eq:harmonics}, then the set
\[
\calY = \{ Y_{l,k}: l \in \bbN_0, k \in \bbK_l \}
\]
is an orthonormal basis of $L^2(\bbS^{n-1})$. In particular for any $f \in L^2(\bbS^{n-1})$ the following identity holds
\[
f =\sum_{l\in \bbN_0} \proj_l f,
\]
in the sense that $\lim_{n\to \infty} \|f -  \sum_{l=1}^n \proj_l f\|_{L_2} =0$.
\end{theorem}

Given an arbitrary spherical harmonics basis one can define the \emph{zonal harmonics}, namely the functions $Z_l: \bbS^{n-1}\times \bbS^{n-1} \to \bbR$ of the form
\begin{equation}
    \label{eq:zonal}
    Z_l(x, y) := \sum_{k\in\bbK_l} Y_{l,k}(x)Y_{l, k}(y).
\end{equation}
Using the formula above we conclude that the projection operator takes the following form in terms of the zonal harmonics:
\[
(\proj_l f)(x) = \int_{\bbS^{n-1}} f(x) Z_l(x, y) d\sigma(y).
\]
As follows from \cite[Lemma 1.2.3]{dai2013approximation}, the zonal harmonics $Z_l$ are independent of the choice of the basis $\calY$; moreover, the following relation holds.
\begin{proposition}[Zonal harmonics {\cite[Theorem 1.2.6]{dai2013approximation}}]
\label{prop:zonal}
    For arbitrary $x, y \in \bbS^{n-1}$ and $l \in \bbN_0$ the zonal harmonics $Z_l$ take the following representation in terms of the Gegenbauer polynomials
    \[
    Z_l(x, y) = \frac{2l + n-2}{n-2}C_l^{\frac{n-2}{2}}(\left<x, y\right>).
    \]
\end{proposition}
Since the expression on the right-hand side is a function of the scalar product $\left<x, y\right>$ only, we will use the notation $Z_l(x, y) = Z_l(\left<x, y\right>)$ interchangeably. 

Similarly, we can define a class of \emph{zonal} kernels.
\begin{definition}[Zonal kernels]
An interaction kernel $W: \bbS^{n-1} \times \bbS^{n-1}\to \bbR$ is \emph{zonal} if it only depends on the scalar product, namely $W(x, y) = W(\left<x, y\right>)$.
\end{definition}
 For a zonal kernel $W$ the convolution operator is defined as follows.

\begin{definition}[Convolution on $\bbS^{n-1}$]
	\label{def:conv-sn}
	Let $f\in L^2(\bbS^{n-1})$ and let $W$ be a zonal kernel satisfying the integrability condition:
	\begin{equation}\label{eq:spherical:intergrability}
		\int_{\bbS^{n-1}} \left|W(x_0, y)\right|d\sigma(y) < \infty,
	\end{equation}
	for any $x_0 \in \bbS^{n-1}$, then the convolution of $f$ with $W$ is defined as
	\[
	(W*f)(x) :=%
	\int_{\bbS^{n-1}} f(y)W(x, y)d\sigma(y).
	\]
\end{definition}
Note that since the  kernel $W$ is zonal, the integrability assumption as above does not depend on the choice of $x_0$. The symmetric structure of a zonal kernel allows to establish the spectral properties of the convolution operator, and this is the subject of the next section. 

\subsection{Spectral properties of the convolution}
\label{ssec:properties}
Recall that on a flat torus, the convolution operator is a diagonal operator in the Fourier basis. Analogously, convolution with a zonal kernel is diagonal in the basis of spherical harmonics. To make this statement concrete, in this section we define the \emph{spherical harmonics decomposition} of a zonal kernel and establish the spectral properties of the convolution operator. We also give a semi-formal calculation in the basis of spherical harmonics in order to give an intuition behind Lemmas \ref{lemma:delta} and \ref{lemma:gradient}.

For a zonal kernel $W$, we define its spherical harmonic decomposition as follows:
\begin{definition}[Spherical harmonics decomposition]
	\label{def:spherical-decomposition}
	Let $W$ be a zonal kernel satisfying the integrability condition~\eqref{eq:spherical:intergrability}. Then the sequence $(\hat W_l)_{l\in \bbN}$ is called the \emph{spherical harmonics decomposition} of $W$, where
	\[
	\hat W_l =\frac{1}{Z_l(x_0, x_0)} \int_{\bbS^{n-1}} W(x_0, y)Z_l(x_0, y)d\sigma(y),
	\]
    and $Z_l$ are the zonal harmonics.
\end{definition}
Note that due to the symmetry, the definition above does not depend on the choice of $x_0$. As follows from Proposition \ref{prop:zonal}, the spherical harmonics decomposition allows to represent any admissible $W$ as a linear combination of Gegenbauer polynomials.
\begin{lemma}[\cite{dai2013approximation}] 
\label{lemma:gegegnbauer-decomposition} 
Let $W$ be a zonal kernel satisfying the integrability condition~\eqref{eq:spherical:intergrability}, then $W$ has the following representation in terms of the Gegenbauer polynomials:
\[
W(x, y) = \sum_l\hat W_l Z_l(x, y) =  \sum_{l} \hat W_l \frac{2l +n-2}{n-2} C_l^{\frac{n-2}{2}}(\left<x, y\right>),
\]
where $(\hat W_l)_{l\in\bbN}$ is the spherical harmonics decomposition of $W$ and the equality holds in $L^2(\bbS^{n-1} \times \bbS^{n-1})$ sense. 
\end{lemma}

\begin{remark}[Defining a kernel by the spherical harmonics decomposition]
\label{remark:kernel-by-decomp}
  Consider a sequence $(a_l)_{l\in\bbN_0}$ and assume that the series
\[
A = \sum_{l=0}^\infty a_lZ_l(\left<x_0, \cdot\right>)
\]
converge in $L^2(\bbS^{n-1})$ sense, then $A(x, y) := \sum_l a_l Z_l(x, y)$ is a zonal kernel with the spherical harmonics decomposition $(a_l)_{l\in\bbN}$. In particular, if a kernel $A$ is positive semi-definite, namely satisfies $a_l \geq 0$, one can define the its 'convolution square root' as
\[
\sqrt[*]{A}(x, y) := \sum_{l}\sqrt{a_l}Z_l(x, y).
\]
We give a rigorous characterization of the 'convolution square root' for a class of the singular kernels in Section~\ref{ssec:estimates}.
\end{remark}
\begin{remark}
	Note that the coefficients $\hat W_l$ are scaled projections of $W(x_0, \cdot)$ onto the spherical harmonics basis functions $Y_{l,0}$ with a specific choice of the basis $\calY$. 
\end{remark}

With the above definition we can formulate the convolution theorem on $\bbS^{n-1}$. 
\begin{theorem}[Convolution theorem on $\bbS^{n-1}$]
\label{cor:convolution-theorem}
Let $f, W$ be as in Definition \ref{def:conv-sn}, then for any $l\in \bbN, \ k \in \bbK_l$ the $(l, k)$-th spherical harmonics coefficient of the convolution $W* f$ satisfies
\[
\left<W*f, Y_{l, k}\right>_{L^2(\bbS^{n-1})} =  \hat W_l\left<f, Y_{l, k}\right>_{L^2(\bbS^{n-1})}.
\]
\end{theorem}
The proof follows from \cite[Theorem 2.1.3]{dai2013approximation}.

\subsection{Admissible interaction kernels}
\label{sec:repulsive}
In this this section we discuss the assumptions on both interaction kernels $W$ and~$V_\varepsilon$. As mentioned in the introduction, we assume both of them to be zonal and satisfy the following regularity properties.
\begin{assumption}[Properties of the fixed interaction]
\label{assum:fixed-kernel}
The fixed interaction kernel~$W$ is zonal and satisfies $W\in C^2(\bbS^{n-1} \times  \bbS^{n-1} )$. In particular this implies
\[
\|\Delta W\|_{L^\infty(\bbS^{n-1})} := \|\Delta W(x_0, \cdot)\|_{L^\infty(\bbS^{n-1})} < \infty
\]
for any $x_0 \in \bbS^{n-1}$
\end{assumption}
We require the family $(V_\varepsilon)_{\varepsilon \in \bbR_+}$ to satisfy the following localization assumption.
\begin{assumption}[Locally repulsive kernels]
 \label{ass:repulsive}   
 Let $(V_\varepsilon)_{\varepsilon \in \bbR_+}$ be a family of zonal interaction kernels in $C^2(\bbS^{n-1}\times \bbS^{n-1})$ and let $(\hat V_{\varepsilon, l})_{l\in \bbN}$ be the coefficients of the spherical harmonics decomposition of $V_\varepsilon$. We say that the family $(V_\varepsilon)_{\varepsilon \in \bbR_+}$ satisfies the \emph{localization assumption} in the limit $\varepsilon\to 0$ if
 \begin{itemize}
 \item $ V_\varepsilon \geq 0$ and  $\|V_
 \varepsilon\|_{L^1} = \int V_\varepsilon(x, \cdot )d\sigma = 1$, for every $\varepsilon\in \bbR_+$ and arbitrary $x \in \bbS^{n-1}$,
 \item the spherical harmonics decomposition of $\hat V_\varepsilon$ is non-negative and uniformly bounded, in the sense that $\exists C > 0: \ \forall \varepsilon >0, \forall l\in\bbN_0:$
 \begin{equation}
    \label{eq:ass:hatVe-bdd}
0\leq \hat V_{\varepsilon, l} \leq C,
 \end{equation}
 and for every $\varepsilon\in \bbR_+$ satisfies $\sum_{l}l^n\hat V_{\varepsilon, l} < \infty$,
 \item the following pointwise convergence of the components of the spherical harmonics decomposition holds
 \[
\hat V_{\varepsilon, l} \to 1 \quad \text{as} \ \varepsilon \to 0,
 \]
 for every $l\in \bbN_0$.
 \end{itemize}
\end{assumption}
In particular, the above assumptions give the following uniform-in-$\varepsilon$ upper bound on the interaction energy.
\begin{lemma}[Bounds on the energy]
\label{l:energy-bound}
Let Assumptions \ref{assum:fixed-kernel} and \ref{ass:repulsive} be satisfied, and let  $\calF_\varepsilon$ be the interaction energy as defined in \eqref{eq:intro-free-energy}. Then there exists a constant $C>0$ such that for any $\rho \in L^2(\bbS^{n-1})\cap \calP(\bbS^{n-1})$ and for any $\e>0$ we have
\begin{equation}
\label{eq:upper-lower-bound-Fe}
-\frac12 \|W\|_{L^\infty(\bbS^{n-1})}
\leq \calF_\e(\rho)
\leq \frac12 \|W\|_{L^\infty(\bbS^{n-1})} + C\|\rho\|_{L^2(\bbS^{n-1})}^2.
\end{equation}
\end{lemma}

\begin{proof}
Writing $\rho = \sum_{l, k} \alpha_{l, k} Y_{l, k}$ we have
\begin{align*}
   \MoveEqLeft\calF_\varepsilon(\rho) = \frac12 \int W(\left<x, y\right>)\rho(x)\rho(y)d\sigma(x)d\sigma(y) + \frac12  \sum_{l, k}\hat V_{\varepsilon, l} \alpha_{l, k}^2 
   \\
   &\leftstackrel{\eqref{eq:ass:hatVe-bdd}}\leq \frac12 \|W\|_{L^\infty} + C\sum_{l, k}\alpha_{l, k}^2 = \frac12 \|W\|_{L^\infty} + C\|\rho\|_{L^2}^2,
\end{align*}
To get the second inequality above we used the uniform bound on $\hat V_{\varepsilon, l}$ and the $L^\infty$-bound on the fixed interaction kernel $W$. Note that by Assumption \ref{ass:repulsive} the constant $\tilde C$ can be chosen independent of $\varepsilon$.

 From the non-negativity of $\hat V_\e$ we similarly obtain the opposite bound
 \[
 \calF_\e(\rho) \geq -\frac12 \|W\|_{L^\infty}.
 \qedhere
 \]
\end{proof}

 Moreover, we impose an additional assumption on the `convolution square root' which guarantees that $\sqrt[*]{V_\varepsilon} * \rho \in \calP(\bbS^{n-1})$ for arbitrary $\rho \in \calP(\bbS^{n-1})$. In particular, this assumption enables us to use Wasserstein bounds for $\sqrt[*]{V_\varepsilon} * \rho$ in Lemma \ref{lem:comp-v}.
\begin{assumption}[Non-negative `convolution square root']
 \label{ass:conv-root} 
There exists $\varepsilon_0>0$ such that for all $\varepsilon\in(0, \varepsilon_0)$, the `convolution square root' $\sqrt[*]{V_\varepsilon}$ as defined in Remark~\ref{remark:kernel-by-decomp} is a non-negative function.
 \end{assumption}

Note that the functional $\calF_\e$ has the following alternative expression in terms of the convolution square root as defined in Remark \ref{remark:kernel-by-decomp}:
 \begin{equation}
 \label{eq:alt-formulation-Fe}
 \calF_\varepsilon(\rho) = \frac{1}{2}\|\sqrt[*]{V_\varepsilon}*\rho\|^2_{L^2(\bbS^{n-1})}  + \frac{1}{2} \int W(x, y)d\rho(x)d\rho(y).
 \end{equation}

 Also note that for a non-negative kernel $V$ by definition, for arbitrary $x\in \bbS^{n-1}$, the $L^1$ norm satisfies 
 \[
 \begin{aligned}
\int |V(x, \cdot)|d\sigma = \int V(x, \cdot)d\sigma = \int(V*Y_{0, 0})(x) = \hat V_0 \int Y_{0, 0}(x)d\sigma(x) = \hat V_0 .
\end{aligned}
\]
As a result, if the family $V\varepsilon$ satisfies Assumptions \ref{ass:repulsive} and \ref{ass:conv-root}, then the kernel $\sqrt[*]{V_\varepsilon}$ is measure preserving for arbitrary $\varepsilon$, namely 
\[
\int \sqrt[*]{V_\varepsilon}(x, \cdot)d\sigma = \int V_\varepsilon(x, \cdot)d\sigma = 1.
\]

\begin{remark}[The heat kernel is admissible]
    We first remark that the set of admissible families $(V_\varepsilon)_{\varepsilon>0}$ is non-empty. For example, both Assumptions \ref{ass:repulsive} and \ref{ass:conv-root} are satisfied for the heat kernel, which admits the following decomposition into Gegenbauer polynomials:
    \begin{equation}
    \label{eq:heat-kernel-Gegenbauer}
    V_\varepsilon(x,y) = -\sum_l e^{-l(l+n-2)\varepsilon}\frac{2l+n-2}{n-2}\frac{\Gamma(\frac{n}{2})}{2\sqrt{\pi^n}}C_l^{\frac{n-2}{2}}(\skp{x,y}).
    \end{equation}
    For more details see \cite{zhao2018exact} and \cite[Section 4.6.4]{shalova2024anoisy-transformer}.
\end{remark}

\begin{remark}[Equivalence to $\bbT$]
\label{rem:scaling}
    Recall the  fact that the the delta function at $0$ defined on the interval $[-\pi, \pi]$ admits a Fourier decomposition of all ones:
    \[
    \delta_0(x) = \sum_{k=0}^\infty 1 \cdot \cos kx.
    \]
    The second part of Assumption \ref{ass:repulsive} can then be interpreted as convergence of the sequence of interaction kernels to the delta measure on a sphere. This remark also dictates the choice of the scaling in this paper. Finally note that $\bbT^1 = \bbS^1$ and thus the spherical harmonics basis on $\bbS^1$ reduces to the classical Fourier basis.
\end{remark}
\begin{remark}[On Assumption \ref{ass:conv-root}]
    Verifying Assumption \ref{ass:conv-root} for a general family~$V_\varepsilon$ might be  challenging. One possible approach relies on the decomposition into Gegenbauer polynomials. Assuming that $V_\varepsilon$ is a smooth function, by Lemma \ref{lemma:uniform} its decomposition into Gegenbauer polynomials converges uniformly, and thus it is sufficient to show that
    \[
   \sum_{l, k}\sqrt{ \hat V_{\varepsilon, l}} \frac{2l +n-2}{n-2} C_l^{\frac{n-2}{2}}(s) \geq 0,
    \]
for all $s\in [-1, 1]$. For example, comparison to the heat kernel might be of use for this. On $\bbS^1$ the kernel is decomposed in the classical Fourier basis and thus the question is similar to establishing positivity of a function from its Fourier series, which is in general an open problem.
\end{remark}

\subsection{Estimates} 
\label{ssec:estimates}
Following the intuition given above, for a family of kernels $(V_\varepsilon)_{\varepsilon\in\bbR_+}$ satisfying Assumption \ref{ass:repulsive} we define a sequence of the `square roots' $(\sqrt[*]{V_\varepsilon})_{\varepsilon\in\bbR_+}$ as in Remark \ref{remark:kernel-by-decomp} by the sequences of square roots of the corresponding coefficients $(\sqrt{V_{\varepsilon, l}})_{l\in\bbN}$ %
\begin{equation}
\label{eq:conv-root-series}
\sqrt[*]{V_\varepsilon}(x_0, \cdot) := \lim_{L\to \infty }\sum_{l\leq L}\sqrt{ \hat V_{\varepsilon, l}} Z_l(x_0, \cdot),
\end{equation}
where the limit is taken in the $L^2( \bbS^{n-1})$ sense as mentioned in the Remark \ref{remark:kernel-by-decomp}. %
In particular, Assumption \ref{ass:repulsive} guarantees that $\sqrt[*]{V_\varepsilon}$ is well-defined, namely that the series above converges.
Using the spectral representation of $\sqrt[*]{V_\varepsilon}$ we obtain the following properties of the `convolution square root' operator under the localization Assumption~\ref{ass:repulsive}.
\begin{lemma}[Convolution square root]
\label{lem:conv-sq-root}
    Let $(V_\varepsilon)_{\varepsilon\in\bbR_+}$ satisfy Assumption \ref{ass:repulsive}, then $\sqrt[*]{V_\varepsilon} \in H^1(\bbS^{n-1} \times \bbS^{n-1})$ for every $\varepsilon > 0$.
\end{lemma}
\begin{proof}
    Denote the partial sum in \eqref{eq:conv-root-series} by $M_L$, namely
    \[
    M_L := \sum_{l\leq L}\sqrt{ \hat V_{\varepsilon, l}} Z_l(x_0, \cdot) = \sum_{l\leq L}\sqrt{ \hat V_{\varepsilon, l}} \frac{2l +n-2}{n-2} C_l^{\frac{n-2}{2}}(\left<x_0, \cdot\right>).
    \]
    Recall that the spherical harmonics are eigenfunctions of the Laplace-Beltrami operator, implying the same for the zonal harmonics, namely
    \[
    \Delta_x Z_l(x_0, x) = \Delta_x \sum_{k\in \bbK_l}Y_{l, k}(x_0)Y_{l, k}(x) =   \sum_{k\in \bbK_l}Y_{l, k}(x_0)\Delta_x Y_{l, k}(x)= \lambda_l Z_l(x_0, x),
    \]
    where 
\[
\lambda_{l} = -l(n+l-2),
\]
is the $l$-th eigenvalue of the Laplace-Beltrami operator. Moreover, by orthogonality of the spherical harmonics, for any two elements of the spherical harmonics basis $Y_{l, k}, Y_{l', k'}, \in \calY$ we obtain
\begin{equation}
\label{eq:orthogonality}
\left< Y_{l, k}, \Delta Y_{l', k'}\right> = -\left< \nabla Y_{l, k}, \nabla Y_{l', k'}\right>= \lambda_{l}\delta_{l, l'}\delta_{k, k'}.
\end{equation}
Combining the above we can bound the $H^1$ norm of $M_L$ as
\[
    \begin{aligned}
    \MoveEqLeft\|M_L\|^2_{H^1(\bbS^{n-1})} = \|M_L\|^2_{L^2(\bbS^{n-1})} + \left\| \sum_{l\leq L}  \hat V_{\varepsilon, l} \nabla Z_l(x_0, \cdot)\right\|^2_{L^2(T\bbS^{n-1})} 
    \\
    &= \|M_L\|^2_{L^2(\bbS^{n-1})} + \sum_{l\leq L} \hat V_{\varepsilon, l}^2 \left\|\sum_{k\in \bbK_l} Y_{l, k}(x_0)\nabla_x Y_{l, k}(x)\right\|_{L^2(T\bbS^{n-1})}^2 \\
    &=\|M_L\|^2_{L^2(\bbS^{n-1})} - \sum_{l\leq L} \hat V_{\varepsilon, l}^2 \lambda_l \sum_{k\in \bbK_l} Y_{l, k}(x_0)^2\left\|\nabla_x Y_{l, k}(x)\right\|_{L^2(T\bbS^{n-1})}^2 \\
    &= \|M_L\|^2_{L^2(\bbS^{n-1})} - \sum_{l\leq L} \lambda_l \hat V_{\varepsilon, l}^2 \|Z_l(x_0, x)\|_{L^2}^2.
    \end{aligned}
    \]
   Using Proposition \ref{prop:zonal}, for every finite $L$ we thus calculate
    \[
    \begin{aligned}
    \MoveEqLeft\|M_L\|^2_{H^1(\bbS^{n-1})} = \sum_{l\leq L}\hat V_{\varepsilon, l} \frac{(2l +n-2)^2}{(n-2)^2} \|C_l^{\frac{n-2}{2}}\|_{L^2}^2 + \sum_{l\leq L}\hat V_{\varepsilon, l} \frac{l(2l +n-2)^3}{(n-2)^2} \|C_l^{\frac{n-2}{2}}\|_{L^2}^2 \\
    &\lesssim \sum_{l\leq L}l^4\hat V_{\varepsilon, l} \|C_l^{\frac{n-2}{2}}\|_{L^2}^2,
    \end{aligned}
    \]
    and since the norm of the Gegenbauer polynomials satisfies 
    \[
    \|C_l^{\frac{n-2}{2}}\|_{L^2}^2 \lesssim \frac{\Gamma(l +n-2)}{l!(l+ (n-2)/2)}\lesssim  l^{n-4},
    \]
    e.g. see \cite[p.774]{abramowitz1968handbook}, we conclude that
    \[
    \|M_L\|^2_{H^1(\bbS^{n-1})} \lesssim \sum_{l\leq L}l^n\hat V_{\varepsilon, l} \leq \sum_{l}l^n\hat V_{\varepsilon, l} < \infty,
    \]
    by Assumption \ref{ass:repulsive} and thus,  $\sqrt[*]{V_\varepsilon} \in H^1$. 
\end{proof}
Under a stronger integrability assumption on the interaction kernel it is possible to define the convolution operator on the space of probability measures.
\begin{proposition}[Measure convolution] 
\label{prop:measure-conv}
Let $V$ be a zonal kernel, for any $x_0 \in \bbS^{n-1}$ satisfying 
\[
\|V(x_0, \cdot)\|^2_{L^2(\bbS^{n-1})} = \int_{\bbS^{n-1}} V(x_0, \cdot)^2 d\sigma < \infty.
\]
For any $\rho\in \calP(\bbS^{n-1})$ denote its coefficients in the basis of spherical harmonics by
\[
\alpha_{l, k} =\int Y_{l, k}(x)d\rho(x),
\]
then the measure convolution satisfies $V * \rho \in L^2(\bbS^{n-1})$ and takes the form
\[
V* \rho  = \sum_{l, k}\hat V_{ l}\alpha_{l, k}Y_{l, k}.
\]
\end{proposition}
\begin{proof}
    To verify that $V * \rho \in L^2(\bbS^{n-1})$ for arbitrary $\rho$ and $\varepsilon$ note that by Jensen's inequality we get
\[
\begin{aligned}
\MoveEqLeft\int \left(\int V(\left<x, y\right>) d\rho(y)\right)^2 d\sigma(x) \leq \int \int\left(V(\left<x, y\right>)\right)^2d\rho(y) d\sigma(x) \\
&=\int \|V(y, \cdot)\|^2_{L^2(\bbS^{n-1})} d\rho(y) = \|V(y, \cdot)\|^2_{L^2(\bbS^{n-1})} < \infty,
\end{aligned}
\]
Since $V*\rho \in L^2(\bbS^{n-1})$ it is equal to its decomposition in the basis of spherical harmonics and the coefficients are
    \[
    \begin{aligned}
    \label{eq:representation}
    (V*\rho)_{l, k} &= \int \int V(\left<x, y\right>)*\rho(y)d\sigma(y) Y_{l, k}(x) d\sigma(x)\\
    &= \int \hat V_{l} Y_{l, k}(y)\rho(y)d\sigma(y) =\hat V_{l}  \int Y_{l, k}d\rho =\hat V_{ l} \alpha_{l, k},
    \end{aligned}
    \]
    hence the result.
\end{proof}

According to Remark~\ref{remark:kernel-by-decomp} for every $\varepsilon \in \bbR_+$, the sequence $\left(\sqrt{\hat V_{\varepsilon, l}}\right)_{l\in\bbN}$ defines a kernel which we denote by $\sqrt[*]{V_\varepsilon}$. As a result, we can also define the measure convolution operator with $\sqrt[*]{V_\varepsilon}$. In particular, as follows from the Proposition \ref{prop:measure-conv}, the convolution $\sqrt[*]{V_\varepsilon} * \rho$ is well-defined for every $\varepsilon \in\bbR_+$ and $\rho\in \calP(\bbS^{n-1})$ and admits the following form.

\begin{corollary}[Measure convolution with the square root] 
\label{prop:conv-root} 
Let $(V_\varepsilon)_{\varepsilon\in\bbR_+}$ be a family of kernels satisfying the localization Assumption \ref{ass:repulsive}, then for arbitrary $\varepsilon > 0$, $\sqrt[*]{V_\varepsilon} * \rho \in L^2(\bbS^{n-1})$ and takes the form
\[
\sqrt[*]{V_\varepsilon} * \rho  = \sum_{l, k}\sqrt{\hat V_{\varepsilon, l}} \ \alpha_{l, k}Y_{l, k}.
\]  
\end{corollary}
In addition if $\rho \in\calP(\bbS^{n-1})\cap L^2(\bbS^{n-1})$, we get the following properties.
\begin{lemma}[Weak convergence to a delta kernel]
\label{lemma:delta}
    Let $(V_\varepsilon)_{\varepsilon \in \bbR_+}$ be a family of interaction kernels satisfying the localization Assumption \ref{ass:repulsive}. Then for any $u \in L^2(\bbS^{n-1})$ the following convergence holds
    \[
    \|u - \sqrt[*]{V_\varepsilon} * u\|_{L^2(\bbS^{n-1})} \to 0, \quad \text{as } \varepsilon \to 0.
    \]
\end{lemma}
\begin{proof}
    By Assumption \ref{ass:repulsive}, every component of the spherical harmonics decomposition $\hat V_{\varepsilon, l}
    $ converges to $1$, implying that the same holds for the square root, namely $\sqrt{\hat V_{\varepsilon, l}} \to 1$. Analogously we conclude that $\sqrt{\hat V_{\varepsilon, l} }$ are uniformly bounded. Thus, expanding the definition of the convolution, we conclude that
    \[
    \begin{aligned}
     \MoveEqLeft \left\|u -  \sqrt[*]{V_\varepsilon} * u\right\|^2_{L^2(\bbS^{n-1})} =   \left\|\sum_{l, k}\left(1-\sqrt{\hat V_{\varepsilon,l }}\right) \alpha_{l, k}Y_{l, k} \right\|^2_{L^2(\bbS^{n-1})}\\
     &= \sum_{l, k}\left(1-\sqrt{\hat V_{\varepsilon,l }}\right)^2 \alpha_{l, k}^2\to 0,
    \end{aligned}
    \]
    by the dominated convergence theorem.
\end{proof}

\begin{lemma}[Gradient estimate]
\label{lemma:gradient}
 Let $(V_\varepsilon)_{\varepsilon \in \bbR_+}$ be a family of interaction kernels satisfying the localization Assumption \ref{ass:repulsive}. Then for any $\varepsilon >0$ and any $u \in L^2(\bbS^{n-1})$ the convolution $\sqrt[*]{V_\varepsilon} * u =: v_{u, \varepsilon}$ is an element of $H^1(\bbS^{n-1})$, and the (weak) gradient of $v_{u, \varepsilon}$ admits the form
    \begin{equation}
      \label{eq:grad} 
    \nabla v_{u, \varepsilon} = \sum_{l, k} \sqrt{\hat V_{\varepsilon, l}}\alpha_{l, k} \nabla Y_{l, k},
    \end{equation}
    where $\alpha_{l, k}$ are the coefficients of the decomposition of $u$ into the spherical harmonics basis, namely $u = \sum_{l, k}\alpha_{l, k} Y_{l, k}$.
\end{lemma}
\begin{proof}
As follows from \eqref{eq:orthogonality}, the gradients of different spherical harmonics are orthogonal in $L^2(T\bbS^{n-1})$. As a result, we conclude that the series 
\[
\sum_{l, k} \sqrt{\hat V_{\varepsilon, l}}\alpha_{l, k} \nabla Y_{l, k}
\]
converge in $L^2(T\bbS^{n-1})$ if and only if
\[
\sum_{l, k} \hat V_{\varepsilon, l}\alpha^2_{l, k} \|\nabla Y_{l, k}\|^2_{L^2(T\bbS^{n-1})} <\infty.
\]
By Assumption \ref{ass:repulsive} the coefficients $\hat V_{\varepsilon, l}$ satisfy $\hat V_{\varepsilon, l} = O(1/l^n)$ as $l\to \infty$, and thus $\hat V_{\varepsilon, l} \|\nabla Y_{l, k}\|^2_{L^2(T\bbS^{n-1})} = \lambda_l\hat V_{\varepsilon, l}= O(l^{2-n}) = O(1)$  for arbitrary $n\geq 2$. As a result, we conclude that for arbitrary $u\in L^2(\bbS^{n-1})$ it holds that
\begin{equation}
\label{eq:l1210}
\sum_{l, k} \hat V_{\varepsilon, l}\alpha^2_{l, k} \|\nabla Y_{l, k}\|^2_{L^2(T\bbS^{n-1})} \leq C\sum_{l, k} \alpha^2_{l, k} <\infty.
\end{equation}

    Given $u\in L^2(\bbS^{n-1})$, consider the approximating sequence $\phi_j := \sum_{l\leq j, k} \sqrt{\hat V_{\varepsilon, l}}\alpha_{l, k} Y_{l, k}$ for $j \in \bbN$ and note that $\phi_j \in C^\infty$. Estimating the $H^1$ norm we obtain
    \[
    \begin{aligned}
    \MoveEqLeft \|v_{u, \varepsilon} - \phi_j\|_{H^1(\bbS^{n-1})} = \|v_{u, \varepsilon} - \phi_j\|_{L^2(\bbS^{n-1})} + \left\|\sum_{l, k} \sqrt{\hat V_{\varepsilon, l}}\alpha_{l, k} \nabla Y_{l, k} - \sum_{l\leq j, k} {\hat V_{\varepsilon, l}}\alpha_{l, k} \nabla Y_{l, k}\right\|_{L^2(T\bbS^{n-1})}\\
    &=\left(\sum_{l> j, k}\hat V_{\varepsilon, l}\alpha_{l, k}^2\right)^{1/2} + \left(\sum_{l>j, k}\alpha_{l, k}^2\hat V_{\varepsilon, l}\|\nabla Y_{l, k}\|^2_{L^2(T\bbS^{n-1})}\right)^{1/2}:= I + II.
    \end{aligned}
    \]
    Since $\hat V_{\varepsilon, l}$ are uniformly bounded and $u\in L^2(\bbS^{n-1})$ we conclude that $I \to 0$. Analogously, the estimate \eqref{eq:l1210} guarantees that $II \to 0$ as the tail of convergent series. 
    Hence, $v_{u, \varepsilon} \in \overline{C^\infty(\bbS^{n-1})}^{H_1}$ and its gradient $\nabla v_{u, \varepsilon}$ obtains the series representation of form~\eqref{eq:grad}.
\end{proof}

As a result, we obtain the key relation of this work, which can be interpreted as a (very) weak form of the integration by parts formula on the sphere.
\begin{corollary}
\label{cor:flow-interchange}
    Under assumptions of Lemma \ref{lemma:gradient}, for any $u \in C^2(\bbS^{n-1})$ it holds that
    \[
    \|\nabla \sqrt[*]{V_\varepsilon} * u \|^2_{L^2(T\bbS^{n-1})} = - \left<V_\varepsilon * u, \Delta u\right> = -\sum_{l, k}\lambda_l\hat V_{\varepsilon, l} \alpha_{l, k}.
    \]
\end{corollary}
\begin{proof}
    Due to the Lemma \ref{lemma:gradient}, we only need to show
    \[
    - \left<V_\varepsilon * u, \Delta u\right> = -\sum_{l, k}\lambda_l \hat V_{\varepsilon, l}\alpha_{l, k}
    \]
    for arbitrary $u\in C^2(\bbS^{n-1})$. We argue analogously to the proof of Lemma \ref{lem:conv-sq-root}. In particular, \eqref{eq:orthogonality} implies that for any $u\in C^2(\bbS^{n-1})$ it holds that
\[
\Delta u = \sum_{l, k} \alpha_{l, k}\Delta Y_{l, k} = \sum_{l, k} \lambda_{l}\alpha_{l, k} Y_{l, k}.
\]
Calculation of the scalar product $\left< V_\varepsilon*u , \Delta u\right>_{L^2(\bbS^{n-1})}$ thus gives
\begin{align*}
    \left< V_\varepsilon*u , \Delta u\right> &= \left< \sum_{l, k} \hat V_{\varepsilon,l} \alpha_{l, k} Y_{l, k}, \sum_{l', k'}\lambda_{l'}\alpha_{l', k'} Y_{l', k'}\right> \notag \\
    &= \sum_{l, k} \sum_{l', k'} \delta_{l, l'}\delta_{k, k'} \hat V_{\varepsilon,l}\lambda_{l'}\alpha_{l, k}\alpha_{l', k'} = \sum_{l, k}  \lambda_{l}\hat V_{\varepsilon,l}\alpha_{l, k}^2. 
\end{align*}
which completes the proof.

\end{proof}
If the family of localized interaction kernels $(V_\varepsilon)_{\varepsilon>0}$ satisfies Assumption~\ref{ass:conv-root}, the following uniform convergence result holds.
\begin{lemma}[Uniform convergence]
\label{lemma:uniform}
Let $u \in C_b(\bbS^{n-1})$, and let $(V_\varepsilon)_{\varepsilon>0}$ be a family of interaction kernels satisfying Assumptions~\ref{ass:repulsive} and \ref{ass:conv-root}, then 
    \begin{equation}
        \label{eq:unif-convergence}
    \sup_{x\in \bbS^{n-1}}\abs[\big]{u(x) - (\sqrt[*]{V_\varepsilon} * u)(x)} \to 0 \quad \text{as } \varepsilon \to 0.
    \end{equation}
    \end{lemma}
    \begin{proof}
    Note that under Assumptions \ref{ass:repulsive} and \ref{ass:conv-root}, the family of kernels $(\sqrt[*]{V_\varepsilon})_{\varepsilon >0}$ is equivalent to a family of probability measures. Since the sphere is a compact set, by Prokhorov's theorem any family of probability measures on $\bbS^{n-1}$ is relatively compact. By uniqueness of the limit and Lemma \ref{lemma:delta} we conclude that $\sqrt[*]{V_\varepsilon}(x,y)\sigma(dy) \stackrel{w}{\to} \delta_x(dy)$. As a result, $(\sqrt[*]{V_\varepsilon} * u)(x) \to u(x)$ at every $x$ for any $u \in C_b$.

    To show that the convergence actually is uniform over $\bbS^{n-1}$, note that any $u\in C_b$ is uniformly continuous with some continuous modulus of continuity $\omega:[0,\infty)\to[0,\infty)$. For fixed $x$, the function $y\mapsto \omega(\dist(x,y))$ is continuous on $\bbS^{n-1}$, and the pointwise convergence result above applies. We then estimate for any $x\in \bbS^{n-1}$
    \begin{align*}
    \abs[\big]{u(x) - (\sqrt[*]{V_\varepsilon} * u)(x)}
    &= \abs[\bigg]{\int_{\bbS^{n-1}} (u(x)-u(y) \sqrt[*]{V_\e}(x,y) \, \sigma(dy))}\\
    &\leq \int |u(x)-u(y)| \sqrt[*]{V_\e}(x,y) \, \sigma(dy))\\
    &\leq \int \omega(\dist(x,y)) \sqrt[*]{V_\e}(x,y) \, \sigma(dy))\\
    &\longrightarrow 0\qquad\text{as $\e\to0$.}
    \end{align*}
    This final convergence is uniform in $x$ because of the zonal nature of $\sqrt[*]{V_\e}$. This proves the uniform convergence~\eqref{eq:unif-convergence}. %
\end{proof}

    Finally, we will require the following property of the convolution on a sphere. 
\begin{proposition}
\label{prop:grad}
    Let $W \in H^1(\bbS^{n-1} \times \bbS^{n-1})$ be a zonal kernel, then for every $v \in H^1(\bbS^{n-1})$ the following formula holds
    \[
    \nabla_x (W*v)(x) =\int W(x, y) \Pi_{xy} \nabla_y v(y) d\sigma(y) %
    \]
    where $\Pi_{xy}:= \Gamma(\gamma_{y\to x})_0^1$ is the parallel transport map from $y$ to $x$ along the corresponding geodesic, see Appendix \ref{app:parallel} for the definition.
\end{proposition}
\begin{proof}
    The result follows directly from a calculation in the proof of~\cite[Proposition 3.9]{bruno2025multiscale}.
\end{proof}

\section{Solutions of PDEs on the sphere}
\label{sec:preliminaries}
\subsection{Wasserstein spaces of probability measures}
As already mentioned in the introduction, both models \eqref{eq:intro-main} and \eqref{eq:intro-diffuse} admit a gradient flow formulation in the space of probability measures. In this section we define the Wasserstein distance in the manifold setting and introduce some preliminary results required in the further analysis.

Given two probability measures $\mu, \nu \in \calP(\bbS^{n-1})$ we denote the set of probability measures on $\calP(\bbS^{n-1} \times \bbS^{n-1})$ with first and second marginals being $\mu$ and $\nu$ respectively as $\Pi(\mu, \nu)$. Then the Wasserstein-$p$ distance on $\calP(\bbS^{n-1})$ is defined as
\begin{equation}
\label{eq:wasserstein}
W_p(\mu, \nu) := \inf_{\pi \in \Pi(\mu, \nu)}\left(\int d(x, y)^pd\pi(x, y)\right)^{1/p},
\end{equation}
where $d(x, y) := \arccos(\left<x, y\right>)$ is the standard distance on the unit sphere. Since $\bbS^{n-1}$ is a compact set for every $p\in \bbN$ the infimum is achieved by at least one measure $\pi \in \Pi$ and thus $\inf$ can be replaced by $\min$.

We will require the following Lemma characterizing the behaviour of the Wasserstein distance under the convolution. 
\begin{lemma}[Wasserstein distance under convolution]
\label{lem:was-bound}
Let $V$ be a non-negative interaction kernel of form $V(x, y)= V(\left<x, y\right>)$ satisfying $\int V(x_0, x) d\sigma(x) = 1$. Then for arbitrary measures $\mu, \nu \in \calP(\bbS^{n-1})$ the following bound is satisfied
\[
W_p(V*\mu, V*\nu) \leq C_0 W_p(\mu, \nu),
\]  
for some $C_0 >0$ independent of $\mu, \nu$.
\end{lemma}
\begin{proof}
    We adapt the proof of \cite[Lemma 5.2]{santambrogio2015optimal} to the spherical domain. In particular, we introduce a change of variables for the spherical convolution and then define an optimal transport plan which gives the desired bound. 
    
    \emph{Step 1: Change of variables.} Let $x_0$ be a pole of the sphere and note that every point $x\in \bbS^{n-1}\backslash \{-x_0\}$ can be uniquely represented as an end point of a geodesic $x = \exp_{x_0} u_x$ for some $u_x \in B(0, 2\pi) \subset T_{x_0}\bbS^{n-1}$, where $B(0, 2\pi)$ denotes the norm ball in $T_{x_0}\bbS^{n-1}$. Define $\tilde\sigma$ to be the pullback measure of the exponential map on $B(0, 2\pi)$; by definition it satisfies $\tilde\sigma (U) = \sigma (\exp_{x_0}(U))$ for any $U \in B(0, 2\pi)$.
    This allows us to rewrite the convolution on the sphere in the following form
    \[
    \begin{aligned}
    \MoveEqLeft (V*\rho)(x_0) = \int_{\bbS^{n-1}} V(\left<x_0, x\right>)\rho(x)d\sigma(x) \\
    &= \int_{T_{x_0}\bbS^{n-1}} V\left(\left<x_0, \exp_{x_0} u_x\right>\right)\rho(\exp_{x_0} u_x)d\tilde\sigma(u_x) \\   
    &= \int_{T_{x_0}\bbS^{n-1}} \tilde V(\|u_x\|)\rho(\exp_{x_0} u_x)d\tilde\sigma(u_x),
    \end{aligned}
    \]
    where we used the symmetry of the interaction kernel.
    
    \emph{Step 2: Transport plan.} Let $\Pi$ be an optimal transport plan between $\mu$ and $\nu$ and define the transport plan $\Pi_V$ such that for any $\phi \in C^\infty(\bbS^{n-1}\times \bbS^{n-1})$ the following relation is satisfied
    \[
    \begin{aligned}
    \MoveEqLeft\int_{\bbS^{n-1}\times\bbS^{n-1}}\phi(x, y)d\Pi_V(x, y) \\
    &= \int_{\bbS^{n-1}\times\bbS^{n-1}}\int_{T_{x}\bbS^{n-1}} \phi\left(\exp_{x}u_x, \exp_{y}\Gamma(\gamma_{x\to y})_0^1u_x\right) \tilde V(\|u_x\|)d\tilde \sigma(u_x) d\Pi(\mu, \nu),
    \end{aligned}
    \]
    where $\Gamma(\gamma_{x\to y})_0^1$ is the parallel transport map as defined in Appendix \ref{app:parallel}.
    Note that by construction, the marginals of $\Pi_V$ are equal to $V*\mu$ and $V*\nu$ respectively. To illustrate this fact, we check for the first marginal
    \[
    \begin{aligned}
    \MoveEqLeft\int\phi(x)d\Pi_V(x, y) = \int_{\bbS^{n-1}\times\bbS^{n-1}}\int_{T_{x}\bbS^{n-1}} \phi\left(\exp_{x}u_x\right) \tilde V(\|u_x\|)d\tilde \sigma(u_x) d\Pi(\mu, \nu)\\
    &= \int_{\bbS^{n-1}}\int_{T_{x}\bbS^{n-1}} \phi\left(\exp_{x}u_x\right) \tilde V(\|u_x\|)d\tilde \sigma(u_x) d\mu(x) \\
    &= \int_{\bbS^{n-1}} (V*\phi)(x) d\mu(x) = \int_{\bbS^{n-1}\times\bbS^{n-1}} V(\left<x, y\right>)\phi(y)d\sigma(y) d\mu(x) \\
    &= \int_{\bbS^{n-1}} \phi(y) (V*\mu)(y)d\sigma(y).
    \end{aligned}
    \]
    \emph{Step 3: Bounding the distance.} Since $\Pi_V$ is a transport plan, and using Lemma \ref{lem:was-geo} we obtain the following bound on the Wasserstein distance
    \[
    \begin{aligned}
        \MoveEqLeft W^p_p(V*\mu, V*\nu) \leq \int_{\bbS^{n-1}\times\bbS^{n-1}} \dist(x, y)^pd\Pi_V \\
        &= \int_{\bbS^{n-1}\times\bbS^{n-1}} \int_{T_{x}\bbS^{n-1}}\dist\left(\exp_{x}u_x, \exp_{y}\Gamma(\gamma_{x\to y})_0^1u_x\right)^p \tilde V(\|u_x\|)d\tilde\sigma 
        (u_x)d\Pi \\
        &\leq C_0^p \int_{\bbS^{n-1}\times\bbS^{n-1}}\int_{T_{x}\bbS^{n-1}} \dist(x, u)^p \tilde V(\|u_x\|)d\tilde\sigma 
        (u_x) d\Pi \\
        &=C_0^p \int_{\bbS^{n-1}\times\bbS^{n-1}}\dist(x, u)^p d\Pi = C_0^p W_p^p(\mu, \nu),
    \end{aligned}
    \]
    where $C_0$ is the absolute constant from Lemma \ref{lem:was-geo}.
\end{proof}

\begin{lemma}[Distance between geodesics]
\label{lem:was-geo}
Let $x, y \in \bbS^{n-1}$ and $v_x\in T_x\bbS^{n-1}$. Consider the curves $\gamma_x$, $\gamma_y: \bbR_+ \to \bbS^{n-1}$ defined as
\[
\gamma_x(t) := \exp_x(tv_x), \quad \gamma_y(t) := \exp_y(tv_y),
\]
where $v_y$ is the parallel transport of $v_x$ along the geodesic $\gamma_{x\to y}$. Then there exists $C_0 >0$ independent of $v_x$ such that
\[
\dist(\gamma_x(t), \gamma_y(t)) \leq C_0\dist(x, y),
\]
for all $t \in \bbR_+$.
\end{lemma}
We provide a proof of the Lemma with $C_0 = 3$ in Appendix \ref{app:geodesics}. We also conjecture that the constant satisfies $C_0 = 1$ and remark that a sharp estimate on $C_0$ would require a more careful treatment of the underlying geometry. To give some preliminary intuition, note that the curves $\gamma_{x}, \gamma_{y}$ form great circles and consider the boundary cases: a) $v_x \parallel \gamma_{x\to y}'$ and b) $v_x \perp \gamma_{x\to y}'$. In the first case the distance is constant $\dist(\gamma_x(t), \gamma_y(t)) = \dist(\gamma_x(0), \gamma_y(0))$ since the trajectories lie on the same great circle. In the second case the distance is maximal at $t = 0$ and oscillates with the period $2\pi$. For more details see Appendix \ref{app:geodesics}.
\subsection{Weak solutions}
We first define the notion of solutions of the evolution equations \eqref{eq:intro-main} and~\eqref{eq:intro-diffuse} and prove existence of  solutions of \eqref{eq:intro-main}. The existence of solutions of~\eqref{eq:intro-diffuse} follows from Theorem \ref{th:main-diff} below.

Recall that for a separable Hilbert space $H$, the space $L^2_{\mathrm{loc}}(0,\infty;H)$ is the space of (equivalence classes of) strongly measurable functions $u:(0,\infty)\to H$ such that for each $T>0$ the norm $\|u\|_{L^2(0,T;H)}$ is finite. Convergence is defined in terms of convergence of each restriction $u|_{[0,T]}$, and a sequence $u^\e$ in $L^2_{\mathrm{loc}}(0,\infty;H)$ is weakly compact for this convergence iff each sequence of norms $\|u^\e\|_{L^2(0,T;H)}$ is bounded independently of $\e$.

\begin{definition}[Weak solution of~\eqref{eq:intro-diffuse}]
\label{def:sol-ade}
A curve $\rho: [0, \infty) \to \calP_{ac}(\bbS^{n-1}) \cap L^2(\bbS^{n-1})$ is a weak solution of the aggregation equation \eqref{eq:intro-main} with initial conditions $\rho_0$ if it satisfies the following properties:
\begin{itemize}
    \item $t\mapsto \rho_t$ is narrowly continuous on $[0,\infty)$,
    \item for almost every $t \geq0$ the measure $\rho_t$ admits a density with respect to the spherical measure $\sigma$, and $\rho \in L^2_{\mathrm{loc}}\left(0,\infty  ; H^1(\bbS^{n-1})\right)$, and
    \item for any $\varphi \in C^2(\bbS^{n-1})$ and all $t \geq 0$ it holds that
\begin{align}
\MoveEqLeft\int_{\bbS^{n-1}} \phi(x) \rho_t(x) d \sigma -\int_{\bbS^{n-1}} \phi(x) \rho_0(x) d \sigma \notag \\
&=- \int_0^t \int_{\bbS^{n-1} \times \bbS^{n-1}}g_x\Big(\nabla_x \varphi(x),  \nabla_x W(x, y)\Big) d \rho_s(y) d \rho_s(x) d r\notag \\
&\qquad-\int_0^t \int_{\bbS^{n-1}}  g_x\Big(\nabla_x \varphi(x),\nabla_x \rho_s(x)\Big) d \rho_s( x) ds.
\label{eq:solADE}
\end{align}
\end{itemize}
\end{definition}
\begin{definition}[Weak solution of \eqref{eq:intro-main}]
    A curve $\rho^{\varepsilon}:[0, \infty) \rightarrow \mathcal{P}(\bbS^{n-1})$ is a weak measure solution to \eqref{eq:intro-main} with initial conditions $\rho_0$ if it satisfies the following properties:
    \begin{itemize}
    \item $t\mapsto \rho^\varepsilon_t$ is narrowly continuous on $[0,\infty)$,
    
    \item for any $\varphi \in C^1(\bbS^{n-1})$ and all $t \geq0$ it holds that
\begin{align}
\MoveEqLeft\int_{\bbS^{n-1}} \varphi(x) d \rho_t^{\varepsilon}(x)-\int_{\bbS^{n-1}} \varphi(x) d \rho_0(x) \notag \\
&=- \int_0^t \int_{\bbS^{n-1} \times \bbS^{n-1}}g_x\Big(\nabla_x \varphi(x),  \nabla_x U_\varepsilon(x, y)\Big) d \rho_s^{\varepsilon}(y) d \rho_s^{\varepsilon}(x) d s, \label{eq:xi}
\end{align}
where $U_\varepsilon(x, y) := W(x, y) + V_{\varepsilon}(x, y)$.
\end{itemize}
\end{definition}
We now prove existence of a weak solution of~\eqref{eq:intro-main} for arbitrary $\varepsilon > 0$.
\begin{proposition}[Existence of solutions of~\eqref{eq:intro-main}]
\label{prop:existence-ae}
  For any $\varepsilon \in \bbR_+$ and for any  $\rho_0\in \calP(\bbS^{n-1})$  there exists a unique weak solution of \eqref{eq:intro-main}, $\rho^\varepsilon: [0, \infty) \to \calP_{ac}(\bbS^{n-1})$, with  initial condition $\rho^\varepsilon(0) = \rho_0$. 
  
  Moreover, there exists a constant $C>0$ such that for any $\e>0$, if $\rho_0\in \calP_{\mathrm{ac}}(\bbS^{n-1})\cap L^2(\bbS^{n-1})$, then $\rho^\e$ satisfies
  \begin{equation}
    \label{eq:conv-bound-existence-theorem}
  \|\sqrt[*]{V_\varepsilon} *\rho^\varepsilon(t)\|^2_{L^2(\bbS^{n-1})} \leq C\bra[\big]{ \|\rho_0\|^2_{L^2(\bbS^{n-1})} + 1}
  \qquad\text{for all }t\geq0.
  \end{equation}
\end{proposition}

\begin{remark}[Related well-posedness results]
The global well-posedness result of Proposition~\ref{prop:existence-ae} is contained in e.g.~\cite{FetecauPatacchini22,fetecau2021well}, but only for initial data with support confined to a hemisphere. The well-posedness results of~\cite{PiccoliRossi13} and~\cite{burger2023porous} both cover this same type of evolution, but only in $\R^n$; in addition, we will need the estimate~\ref{eq:conv-bound-existence-theorem}, and therefore we give the details of the proof, following that of~\cite{burger2023porous}. 
\end{remark}

\begin{proof}[Proof of Proposition \ref{prop:existence-ae}]
We first prove the uniqueness. Denote the measure-dependent vector-field in the continuity equation \eqref{eq:xi} by $\xi_\varepsilon[\mu] \in T\bbS^{n-1}$:
\[
\xi_\varepsilon[\mu](x): = \nabla_x U_\varepsilon(x, \cdot) * \mu,
\]
and note that for every $\e>0$ the map $x\mapsto \xi_\varepsilon[\mu](x)$ is bounded and Lipschitz continuous uniformly in $\mu$:
\begin{align*}
\|\xi_\varepsilon[\mu](x)\|^2_{L^\infty} &=\sup_{x\in\bbS^{n-1}} g_x(\xi_\varepsilon[\mu](x), \xi_\varepsilon[\mu](x)) \\
&\leq \sup_{x, y\in\bbS^{n-1}} g_x(\nabla_x U_\e(x, y), \nabla_x U_\e(x, y)) < \infty, \\
\|\xi_\varepsilon[\mu](x) - \Pi_{xy}\xi_\varepsilon[\mu](y)\|_{g_x} 
&= \norm[\bigg]{\int_{\bbS^{n-1}}\pra[\big]{\nabla_x U_\e(x, z) - \Pi_{xy}\nabla_y U_\e(y, z)}d\mu(z)}_{g_x} \leq \Lip(\nabla_x U_\varepsilon),
\end{align*}
where $\Lip(\nabla_x U_\varepsilon)$ is the Lipschitz constant of $\nabla_x U_\varepsilon$, namely the smallest  constant satisfying
\begin{multline*}
\|\nabla_x U_\varepsilon(x, z_1) - \Pi_{xy}\nabla_y U_\varepsilon(y, z_2)\|_{g_x} \leq \Lip(\nabla_x U_\e) (\dist(x, y) + \dist(z_1, z_2)) \\
\qquad\text{for all }x,y, z_1, z_2\in \bbS^{n-1}.
\end{multline*}
Note that  the Lipschitz constant is well-defined for all $\varepsilon > 0$ since both kernels $W$ and $V_\varepsilon$ are at least of $C^1$ regularity. 
In addition, $\xi_\varepsilon[\mu]$ is Lipschitz continuous as a function of $\mu$ in the Wasserstein-1 topology:
\[
\begin{aligned}
\MoveEqLeft \left\|\xi_\varepsilon[\mu_1] - \xi_\varepsilon[\mu_2]\right\|^2_{L^\infty} := \sup_{x\in\bbS^{n-1}} \norm[\big]{\xi_\varepsilon[\mu_1](x) - \xi_\varepsilon[\mu_2](x)}_{g_x}^2 \\
&= \sup_{x\in\bbS^{n-1}} \norm[\bigg]{\int_{\bbS^{n-1}}\nabla_x U_\e(x, z_1) d\mu_1(z_1) - \int_{\bbS^{n-1}}\nabla_x U_\e(x, z_2) d\mu_2(z_2)}_{g_x}^2 \\
&\leq \Lip(\nabla_x U_\varepsilon)^2 \left(\int_{\bbS^{n-1}\times\bbS^{n-1}} \dist(z_1, z_2)d\pi(z_1,z_2)\right)^2 = \Lip(\nabla_x U_\varepsilon)^2 W_1(\mu_1, \mu_2)^2,
\end{aligned}
\]
where $\pi$ is an optimal Wasserstein-1 transport plan between $\mu_1$ and $\mu_2$. Thus, the uniqueness of solutions of \eqref{eq:intro-main} follows from a standard Dobrushin argument along the lines of \cite[Theorem A.4]{bruno2024emergence}. 

\medskip
We now turn to the existence. 
     We use the minimizing movement scheme on the space of probability measures $(\calP(\bbS^{n-1}), W_2)$ equipped with the Wasserstein distance to establish existence of weak solutions to \eqref{eq:intro-main}. The proof closely follows the approach of Prop.~3.1 and Th.~3.1 of~\cite{burger2023porous} with differences arising from the lack of the vector structure of the underlying space. 

    \emph{Step 1: Constructing $\rho^\varepsilon$.} For $\tau > 0$ let $ \rho_\tau^\varepsilon: [0, \infty) \to \calP(\bbS^{n-1})$ be the piecewise-constant interpolant obtained as a solution of the minimizing movement scheme in the Wasserstein space $(\calP(\bbS^{n-1}),W_2)$ defined in the following way:
\[
\begin{aligned}
&\rho_\tau^\varepsilon(s)= \rho_{\tau,k}^\varepsilon, \quad \text{for } s \in [k\tau, (k+1)\tau), \quad k=0, 1,\dots, \qquad \qquad \rho^\e_{\tau,0} = \rho_0,\\
&\rho^\varepsilon_{\tau,k} \in \argmin_{\rho\in \calP(\bbS^{n-1})} \calF_\varepsilon(\rho) + \frac{1}{2\tau}W^2_2(\rho, \rho^\varepsilon_{\tau,k-1}),
\end{aligned}
\]
where $\calF_\varepsilon: \calP(\bbS^{n-1}) \to \bbR$ is the energy functional defined in~\eqref{eq:intro-free-energy}.
    Applying the same arguments as \cite[Proposition 3.1]{burger2023porous}, we conclude that the sequence $\rho^\e_\tau$  is weakly compact in the compact-open topology; more precisely, there exists a sequence $\tau_\ell \to 0$ and a weakly continuous curve $\rho^\varepsilon: [0, \infty) \to \calP(\bbS^{n-1})$ such that for every $T>0$, the sequence $\rho_{\tau, k}^\varepsilon\big|_{[0, T]}$ converges weakly to $\rho^\varepsilon\big|_{[0,T]}$, uniformly on the interval $[0, T]$.
    Moreover, by construction, the sequence $(\rho_{\tau, k}^\varepsilon)_{k\in \bbN}$ satisfies
    \begin{equation}
        \label{eq:estimate:cont}
    \frac 1{2\tau}\sum_{k\geq 0} {W_2^2(\rho^\varepsilon_{\tau, k},\rho^\varepsilon_{\tau, k-1}) } 
    \leq \calF_{\varepsilon}(\rho_0) - \inf_\rho \calF_\e (\rho)    
    \stackrel{\text{Lem.~\ref{l:energy-bound}}}\leq C(1+\|\rho_0\|^2_{L^2(\bbS^{n-1})}).
    \end{equation}
   Note that this bound also implies the following uniform-in-$\e$ continuity estimate:
    \begin{align}
    \MoveEqLeft W_2(\rho^\varepsilon_{\tau}(s),\rho^\varepsilon_{\tau}(t)) \leq  \sum_{k = \lfloor\frac{s}{\tau}\rfloor}^{\lfloor\frac{t}{\tau}\rfloor }W_2(\rho^\varepsilon_{\tau, k},\rho^\varepsilon_{\tau, k-1}) \leq \left(\frac{t-s}{\tau} + 1 \right)^{1/2}\left(\sum_{k = \lfloor\frac{s}{\tau}\rfloor}^{\lfloor\frac{t}{\tau}\rfloor }W^2_2(\rho^\varepsilon_{\tau, k},\rho^\varepsilon_{\tau, k-1})\right)^{1/2} \notag \\
    &\leq c (\sqrt{\tau}+\sqrt{t-s}), \label{eq:weak-continuity}
    \end{align}
    for some positive constant $c>0$.

    \emph{Step 2: Perturbing the interaction energy.} For every $\tau, \varepsilon$ consider the sequence $(\rho^\varepsilon_{\tau,k})_{k\in \bbN}$ constructed in Step 1. For any $\varphi \in C^\infty(\bbS^{n-1})$ and $\eta > 0$ introduce the perturbation of $\rho^\varepsilon_{\tau,k}$ of form
    \[
    \rho^\eta := (\exp_x \eta \nabla \varphi(x))_\# \rho^\varepsilon_{\tau,k},
    \]
where $(F)_\# \rho$ is the push-forward of $\rho$ under the map $F$. Estimating the difference $\calF_\varepsilon(\rho^\eta) - \calF_\varepsilon(\rho^\varepsilon_{\tau,k})$ we obtain
\[
\begin{aligned}
   \MoveEqLeft\frac{1}{\eta}\left(\calF_\varepsilon(\rho^\eta) - \calF_\varepsilon(\rho^\varepsilon_{\tau,k})\right) =  \frac{1}{2\eta}\iint U_{\varepsilon}(x, y)d\rho^\eta(x)d\rho^\eta(y)
    -\frac{1}{2\eta}\iint U_{\varepsilon}(x, y)d\rho^\varepsilon_{\tau,k}(x)d\rho^\varepsilon_{\tau,k}(y) \\
   &=\iint \frac{\left(U_\varepsilon(\exp_x \eta \nabla \varphi(x), \exp_y \eta \nabla \varphi(y)) - U_\varepsilon(x, y)\right)}{2\eta}d\rho^\varepsilon_{\tau,k}(x)d\rho^\varepsilon_{\tau,k}(y) \\
   &= \iint \frac{\eta g_x\left(\nabla_xU_\varepsilon(x,y),  \nabla \varphi(x)\right) + o(\eta)}{\eta}d\rho^\varepsilon_{\tau,k}(x)d\rho^\varepsilon_{\tau,k}(y), 
\end{aligned}
\]
where the last equality follows from the symmetry of the interaction kernels. Note that the pointwise convergence
\[
\frac{1}{\eta}\left(U_\varepsilon(\exp_x \eta \nabla \varphi(x), y) - U_\varepsilon(x, y)\right) \to g_x\left(\nabla_xU_\varepsilon(x,y),  \nabla \varphi(x)\right)
\]
holds for every $x,y \in\bbS^{n-1}$ and $\phi \in C^\infty(\bbS^{n-1})$ by the definition of the gradient. Hence, by means of the dominated convergence theorem we conclude that
\[
\frac{1}{\eta}\left(\calF_\varepsilon(\rho^\eta) - \calF_\varepsilon(\rho^\varepsilon_{\tau,k})\right)  \stackrel{\eta\to 0}{\to} \iint g_x\left(\nabla_x U_\varepsilon(x,y),  \nabla \varphi(x)\right)d\rho^\varepsilon_{\tau,k}(x)d\rho^\varepsilon_{\tau,k}(y),
\]
where we used that for every $\varepsilon >0$ and arbitrary $\varphi \in C^\infty(\bbS^{n-1})$ the fraction
\[
\frac{U_\varepsilon(\exp_x \eta \nabla \varphi(x), \exp_y \eta \nabla \varphi(y)) - U_\varepsilon(x, y)}{2\eta}
\]
is bounded uniformly in $\eta$ and $x, y \in \bbS^{n-1}$ since, by the regularity assumptions on the kernels $V_\varepsilon$ and $W$, their sum $U_\varepsilon$ is $C^1$ on $\bbS^{n-1} \times \bbS^{n-1}$. %

\emph{Step 3: Perturbing the Wasserstein distance.}
Let $\gamma^\varepsilon_{\tau, k}$ be an optimal transport plan between $\rho^\varepsilon_{\tau,k-1}$ and $\rho^\varepsilon_{\tau,k}$. Estimating the change of the Wasserstein distance under the same perturbation of $\rho^\varepsilon_{\tau,k}$ as in Step 2, we obtain:
\begin{align*}
    \MoveEqLeft\frac{1}{2\tau}\left(\frac{W^2_2(\rho^\eta, \rho^\varepsilon_{\tau,k-1}) - W^2_2(\rho^\varepsilon_{\tau,k}, \rho^\varepsilon_{\tau,k-1})}{\eta}\right) \\
    &\leq \frac{1
    }{2\tau\eta}\iint \left(\dist^2(x, \exp_y{\eta\nabla \varphi(y)}) - \dist^2(x, y)\right)d\gamma^\varepsilon_{\tau, k}(x, y) \\
    &\xrightarrow{\eqref{eq:deriv-of-dsquared}}  -\frac{1
    }{\tau}\iint g_y(\log_y x, \nabla \varphi(y)) \, d\gamma^\varepsilon_{\tau, k}(x, y) \qquad\text{as }\eta\to0.
\end{align*}
    
\emph{Step 4: Combining the estimates.}
Since $\rho_{\tau, k}^\varepsilon$ is a solution of the minimizing movement scheme, the following inequality holds for arbitrary $\eta>0$ and $\varphi$:
\[
\calF_\varepsilon(\rho^\eta) + \frac{1}{2\tau}W^2_2(\rho^\eta, \rho^\varepsilon_{\tau,k-1}) \geq \calF_\varepsilon(\rho^\varepsilon_{\tau,k}) + \frac{1}{2\tau}W^2_2(\rho^\varepsilon_{\tau,k}, \rho^\varepsilon_{\tau,k-1}).
\]
After rearranging we obtain for $\eta>0$
\[
\frac{1}{2\tau}\left(\frac{W^2_2(\rho^\eta, \rho^\varepsilon_{\tau,k-1}) - W^2_2(\rho^\varepsilon_{\tau,k}, \rho^\varepsilon_{\tau,k-1})}{\eta}\right) \geq -\frac{1}{\eta}\left(\calF_\varepsilon(\rho^\eta) - \calF_\varepsilon(\rho^\varepsilon_{\tau,k})\right). 
\]
Hence, taking $\eta \to 0$ we obtain
\[
-\frac{1
    }{\tau}\iint g_y(\log_y x, \nabla \varphi(y)) d\gamma^\varepsilon_{\tau, k}(x, y) 
    \geq 
    -\iint g_x\left(\nabla_x U_\varepsilon(x,y),  \nabla \varphi(x)\right)d\rho^\varepsilon_{\tau,k}(x)d\rho^\varepsilon_{\tau,k}(y)
\]
Replacing $\varphi$ by $-\varphi$ gives the equality
\[
\frac{1
    }{\tau}\iint g_y(\log_y x, \nabla \varphi) d\gamma^\varepsilon_{\tau, k}(x, y) = \iint g_x\left(\nabla_x U_\varepsilon(x,y),  \nabla \varphi(x)\right)d\rho^\varepsilon_{\tau,k}(x)d\rho^\varepsilon_{\tau,k}(y).
\]
Moreover, by definition of the manifold gradient for any $\varphi \in C^\infty$ we obtain:
\[
g_y(\log_y x, \nabla \varphi) = \varphi(x) - \varphi(y) + O(\dist^2(x, y)),
\]
uniformly in $x, y$ as $\dist(x, y) \to 0$, which implies that
\[
\begin{aligned}
\frac{1
    }{\tau}\iint g_y(\log_y x, \nabla \varphi) d\gamma^\varepsilon_{\tau, k}(x, y) &=\frac{1
    }{\tau}\int \varphi(x)\left(d\rho^\varepsilon_{\tau, k} - d\rho^\varepsilon_{\tau, k-1}\right) 
    + O\left(\frac1\tau W_2^2(\rho^\varepsilon_{\tau, k}, \rho^\varepsilon_{\tau, k-1})\right).
    \end{aligned}
\]
Multiplying by $\tau$ and summing over the time steps we obtain
\begin{multline*}
\int \varphi(x)\left(d\rho^\varepsilon_{\tau}(T) - d\rho^\varepsilon_{\tau}(0)\right) 
\\
=\sum_k\iint g_x\left(\nabla_x U_\varepsilon(x,y),  \nabla \varphi(x)\right)d\rho^\varepsilon_{\tau,k}(x)d\rho^\varepsilon_{\tau,k}(y) 
+ O\left(\sum_k W_2^2(\rho^\varepsilon_{\tau, k}, \rho^\varepsilon_{\tau, k-1})\right). 
\end{multline*}
Note that by Lemma~\ref{l:energy-bound} for any $\rho_0 \in L^2(\bbS^{n-1})$ the energy $\calF_{\varepsilon}(\rho_0)$ is bounded uniformly in $\varepsilon$. Hence, using the estimate \eqref{eq:estimate:cont} we conclude that the error term satisfies
\[
O\left(\sum_k W_2^2(\rho^\varepsilon_{\tau, k}, \rho^\varepsilon_{\tau, k-1})\right) = O(\tau),
\]
and taking the limit $\tau_\ell \to 0$ we conclude that $\rho^\varepsilon$ is a weak solution of~\eqref{eq:intro-main}.

\emph{Step 5: $L^2$ bound.}  We now prove the bound~\eqref{eq:conv-bound-existence-theorem} under the additional assumption that $\rho_0\in L^2$. By construction of $\rho_{\tau,k}^\varepsilon$ we obtain  
    \[ \calF_\varepsilon(\rho^\varepsilon_{\tau,k}) + \frac{1}{2\tau}W^2_2(\rho, \rho^\varepsilon_{\tau,k-1}) \leq \calF_\varepsilon(\rho^\varepsilon_{\tau,k-1}),
    \]
    and after rearranging, iterating over $k$ and using the form~\eqref{eq:alt-formulation-Fe} for $\calF_\varepsilon$,   we obtain
    \[
    \begin{aligned}
\MoveEqLeft\frac{1}{2}\|\sqrt[*]{V_\varepsilon}*\rho^\varepsilon_{\tau,k}\|^2_{L^2(\bbS^{n-1})} \leq \frac{1}{2}\|\sqrt[*]{V_\varepsilon} *\rho^\varepsilon_{\tau,0}\|^2_{L^2(\bbS^{n-1})}\\
&\quad+ \frac{1}{2} \int W(x, y)(d\rho^\varepsilon_{\tau,0}(x)d\rho^\varepsilon_{\tau,0}(y) - d\rho^\varepsilon_{\tau,k}(x)d\rho^\varepsilon_{\tau,k}(y)) \\
&\leq \frac12 \|\sqrt[*]{V_\varepsilon} *\rho_0\|^2_{L^2(\bbS^{n-1})}  + \|W\|_{L^\infty}.
    \end{aligned}
    \]
    Since the bound is independent of $\tau$, passing to the limit $\tau \to 0$ we conclude that $\|\sqrt[*]{V_\varepsilon} *\rho^\varepsilon(t)\|^2_{L^2(\bbS^{n-1})} \lesssim \|\sqrt[*]{V_\varepsilon} *\rho_0\|^2_{L^2(\bbS^{n-1})} + \|W\|_\infty$. Moreover, by Assumption \ref{ass:repulsive}, there exists $C>0$ such that for all $\e$ and all $\rho_0\in L^2$, 
    \[
    \|\sqrt[*]{V_\varepsilon} *\rho_0\|^2_{L^2(\bbS^{n-1})} \leq C\|\rho_0\|^2_{L^2(\bbS^{n-1})}.
    \]
    This proves the bound~\eqref{eq:conv-bound-existence-theorem}.
\end{proof}

\subsection{Heat flow on $\bbS^{n-1}$}
The compactness argument in Lemma \ref{lem:comp-v} below relies on the flow interchange technique introduced in \cite{matthes2009family}, where the auxiliary flow is the heat flow. The same argument was also used in the Euclidean setting in \cite{burger2023porous}. In this section we give a concise characterization of the heat flow on $\bbS^{n}$ following~\cite{erbar2010heat}.

\begin{definition}[Heat flow]
\label{def:heat-flow}
    The heat flow on a sphere is the unique semigroup $\calS^t$ generating gradient flow solutions of the relative entropy $\calE: \calP(\bbS^{n-1}) \to \bbR$ in $W_2$ topology, where $\calE$ is defined as
    \begin{equation}
\label{eq:entropy}
\calE(\mu) := \begin{cases}
    \int_{\calM} \rho \log \rho \,d{\sigma} & \text{ if } \mu \text{ admits density } \rho \text{ w.r.t. } \sigma,  \\
    +\infty &\text{otherwise.}
\end{cases}
\end{equation}
\end{definition}
The uniqueness of $\calS^t$ is proved in \cite[Theorem 1]{erbar2010heat}. Moreover, from~\eqref{eq:heat-kernel-Gegenbauer} it follows that for any $\rho \in \calP(\bbS^{n-1})$  the action of the heat semigroup takes the form 
    \begin{equation}
        \label{eq:heat-sg-in-SH-form}
    \calS^s\rho = \sum_{ l, k}e^{-sl(n-2+l)}\alpha_{l, k}Y_{l,k}
    \qquad \text{where $\alpha_{l,k}= \langle \rho,Y_{l,k}\rangle$.} 
    \end{equation}

In addition the semigroup $\calS^t$ 
satisfies the Evolution Variational Inequality (EVI).
\begin{proposition}[Evolution Variational Inequality (EVI)]
    \label{prop:evi}
    For all $\rho_0, \nu \in \calP(\bbS^{n-1})$ such that $\calE(\nu)<\infty$, the following inequality is satisfied:
    \begin{equation}
        \label{eq:EVI}
    \frac{1}{2}\frac{d^+}{dt}W_2^2(\calS^t\rho_0, \nu) \leq \calE(\nu) - \calE(\calS^t\rho_0) -\frac{n-2}{2}W_2^2(\calS^t\rho_0, \nu),
    \end{equation}
    for all $t\in [0, \infty)$.
\end{proposition}

\begin{proof}
For the case $\calE(\rho_0)<\infty$ this result is~\cite[Remark 4.5]{erbar2010heat}, where 
the factor $n-2$ is the Ricci curvature of the sphere. The result can then be extended to arbitrary $\rho_0\in\calP(\bbS^{n-1})$ as described in~\cite[Remark~3.4]{muratori2020gradient}. 
\end{proof}

We will also require the following property of the heat flow.
\begin{lemma}
\label{lemma:heat}
    Let $\rho \in \calP(\bbS^{n-1})$ and let $V \in L^2$ be a zonal kernel. Then for any $\rho\in\calP(\bbS^{n-1})$
    \[
    V* \calS^s\rho \to V* \rho \quad  \text{ in  $L^2(\bbS^{n-1})$ as }s\downarrow 0.
    \]
\end{lemma}
\begin{proof}
    From Proposition \ref{prop:measure-conv} recall that $V*\rho$ is an element of $L^2(\bbS^{n-1})$ and admits the following decomposition in the basis of spherical harmonics:
    \[
    \begin{aligned}
    \label{eq:representation}
    V*\rho =\sum_{l, k}\hat V_{l} \alpha_{l, k} Y_{l, k},
        \qquad \text{where $\alpha_{l,k}= \langle \rho,Y_{l,k}\rangle$.} 
    \end{aligned}
    \]
Hence, we obtain from~\eqref{eq:heat-sg-in-SH-form} that 
    \[
    \|V* \calS^s\rho - V* \rho\|^2_{L^2} =  \sum_{l, k}\alpha_{l, k}^2 \hat V_{l, k}^2(1 - e^{-sl(n-2+l)})^2 \to 0
    \qquad\text{as } s\downarrow 0,
    \]
    which concludes the proof.
\end{proof}
\subsection{Other auxiliary results}
We will also require the following adaptation of the Aubin-Lions lemma for the case when the direct embedding for the derivative is not available. 
\begin{proposition}[{\cite[Theorem 2]{rossi2003tightness}}]
\label{prop:savare}
    Let $X$ be a separable Banach space and consider a family $\Lambda$ of $X$-valued measurable functions. Assume that there exists a lower-semicontinuous functional $\calF: X \to \bbR_\infty$ with compact sublevel sets. In addition, assume that there exists a semi-norm $g$ compatible with $\calF$ in the sense that for all $u, v: \calF(u), \calF(v) < \infty$ it holds that $g(u, v) = 0\  \Rightarrow \ u = v$ a.e. on $[0, T]$. If the family $\Lambda$ satisfies the following two conditions:
    \begin{itemize}
        \item (compactness in space) 
        \[
        \sup_{u\in \Lambda} \int_0^T \calF(u(t))dt < \infty
        \]
        \item (equicontinuity)
        \[
        \lim_{h\downarrow 0}\sup_{u\in \Lambda}\int_0^{T-h} g(u(t+h), u(t))dt = 0,
        \]
    \end{itemize}
    then it is relatively compact in measure on $[0, T]\times X $.
\end{proposition}
We remark that both Proposition \ref{prop:savare} and the Aubin-Lions lemma rely on the combination of the compactness in space (tightness) and equicontinuity arguments and thus may be interpreted as refined versions of the Arzela-Ascoli theorem.

\section{Main result}
\label{sec:convergence}
Given the family $\rho^\varepsilon$ of weak solutions to \eqref{eq:intro-main}, constructed in Proposition \ref{prop:existence-ae}, we construct the corresponding family of spatially regularized curves, $v^\varepsilon: [0, \infty) \to \calP(\bbS^{n-1})$ defined as
\[
v^\varepsilon(t) := \sqrt[*]{V_\varepsilon} *\rho^\varepsilon(t),
\]
for all $t\geq0$. Following the approach of \cite{burger2023porous}, we prove that both families $(\rho^\varepsilon)_{\varepsilon \in\bbR_+}$ and $(v^\varepsilon)_{\varepsilon \in\bbR_+}$ are compact in appropriate topologies. We then show that limits of $\rho^\varepsilon$ and $v^\varepsilon$ coincide, and that every limit point is a weak solution to \eqref{eq:intro-diffuse}.

We now present our main result, Theorem \ref{th:main-diff}, but postpone the proof to Section~\ref{sec:proof-main} which comes after the compactness arguments proved in Section~\ref{sec:compactness}.

\begin{theorem}[Convergence of \ref{eq:intro-main} to \eqref{eq:intro-diffuse}]
\label{th:main-diff}  
Let the interaction kernels $W$, $V_\varepsilon$ satisfy Assumptions \ref{assum:fixed-kernel}, \ref{ass:repulsive} and \ref{ass:conv-root}. Let $(\rho^\varepsilon)_{\varepsilon \in \bbR_+}$ be a family of weak solutions of~\eqref{eq:intro-main} with $\rho_0 \in L^2(\bbS^{n-1})\cap \calP(\bbS^{n-1})$. Then there exists a subsequence $\rho^{\varepsilon_k}$ and a weak solution $\rho$ of~\eqref{eq:intro-diffuse} such that $\rho^{\varepsilon_k}$  converges to $\rho$. The type of convergence is specified in Lemmas~\ref{lem:comp-rho} and~\ref{lem:comp-v} below.
\end{theorem}

\begin{remark}[Uniqueness of solutions of~\eqref{eq:intro-diffuse}]
    The question of uniqueness of solutions of~\eqref{eq:intro-diffuse} is subtle. In general, weak solutions may not be unique, and an entropy condition may be necessary to obtain uniqueness (see e.g.~\cite{Carrillo99,BurgerCapassoMorale2007}). Burger, Capasso, and Morale~\cite{BurgerCapassoMorale2007} prove existence and uniqueness of entropy solutions for similar equations in flat space, and we conjecture that similar results hold on the sphere. 
\end{remark}

\subsection{Compactness of $\rho^\varepsilon$ and $v^\varepsilon$}
\label{sec:compactness}
\begin{lemma}[Compactness of $\{\rho^\varepsilon\}$]
\label{lem:comp-rho}
    Let $\{\rho^\varepsilon\}_{\varepsilon>0}$ be a family of weak solutions of~\eqref{eq:intro-main}, then there exists a subsequence $\rho^{\varepsilon_k}$ and a weakly continuous curve $\rho: [0, T] \to \calP(\bbS^{n-1})$ such that $\rho^{\varepsilon_k}(t) \longrightharpoonup\rho(t)$ for all $t\in [0, T]$.
\end{lemma}
\begin{proof}
    We adapt the arguments of the proof of \cite[Proposition 4.1]{burger2023porous}. Since the stability of optimal transport plans \cite[Theorem 5.20]{Villani2008} holds on arbitrary Polish spaces as well as the used version of Arzela-Ascoli lemma \cite[Proposition 3.3.1]{ambrosio2005gradient}, in view of the estimate \eqref{eq:weak-continuity} we get the result.
\end{proof}

\begin{lemma}[Compactness of $\{v^\varepsilon\}$]
    \label{lemma:bound-v}
    Let $\rho_0 \in L^2(\bbS^{n-1}) \cap \calP(\bbS^{n-1})$, and let $\rho^\e$ be a family of solutions of \eqref{eq:intro-main} with the initial condition $\rho_0$. Set $v^\e := \sqrt[*]{V_\e}{*}\rho^\e$. Then  there exists a constant $C$ such that for any $T\in \bbR_+$ and $\e>0$ we have 
    \[
    \|v^\varepsilon\|_{L^2(0, T; H^1(\bbS^{n-1}))} \leq  CT.
    \]
    Moreover, for any sequence $\varepsilon_k \to 0$ there exists a subsequence $\e_{k_\ell}$ and a curve $\tilde v \in L^2_{\mathrm{loc}}(0,\infty, H^1(\bbS^{n-1}))$ such that  for each $T>0$ we have $v^{\varepsilon_{k_\ell}} \stackrel{w}{\to} \tilde v$ in $L^2(0, T; H^1(\bbS^{n-1}))$.
\end{lemma}
\begin{proof}
    Throughout this proof we fix the final time $T>0$.
    Consider the sequence of interpolants $\rho_{\tau, k}^\varepsilon$ constructed in the proof of Prop.~\ref{prop:existence-ae}. To bound the $L^2(0,T; L^2(\bbS^{n-1}))$ norm of $v^\varepsilon$ note that
    \[
    \begin{aligned}      
    \MoveEqLeft\|v^\varepsilon_{\tau, k}\|_{L^2(0, T;  L^2(\bbS^{n-1}))}^2 = \int_0^T \iint V_\varepsilon(x, y) \rho_{\tau, k}^\varepsilon(t)(x) \rho_{\tau, k}^\varepsilon(t)(y)d\sigma(x)d\sigma(y)dt \\
    &= \int_0^T \calF_\varepsilon(\rho_{\tau, k}^\varepsilon(t)) - \iint W(x, y) \rho_{\tau, k}^\varepsilon(t)(x) \rho_{\tau, k}^\varepsilon(t)(y)d\sigma(x)d\sigma(y)dt \\
    &\leq T\calF_\varepsilon(\rho_0) + T\|W\|_\infty \leq CT(\|\rho_0\|^2_{L^2(\bbS^{n-1})}+1) + T\|W\|_{L^\infty},
    \end{aligned}
    \]
    since $\calF_\varepsilon(\rho_0)$ is decreasing along every curve $\rho^\varepsilon_{\tau, k}$. The sequence $v^\varepsilon_{\tau, k}$ is bounded in $L^2(0, T;  L^2(\bbS^{n-1}))$ and thus by the Banach-Alaoglu theorem there exists a weakly convergent subsequence and a curve $\tilde v^\varepsilon$ such that $v^\varepsilon_{\tau_k} \to \tilde v^\varepsilon$ weakly in $L^2(0, T;  L^2(\bbS^{n-1}))$. By uniqueness of the limit we conclude that $\tilde v^\varepsilon = \sqrt[*]{V_\varepsilon} * \tilde \rho^\varepsilon$ and since the norm is lower-semicontinuous we obtain the following bound:
    \[
    \|\tilde v^\varepsilon\|_{L^2(0, T;  L^2(\bbS^{n-1}))}^2  \leq CT(\|\rho_0\|^2_{L^2(\bbS^{n-1})}+1) + T\|W\|_{L^\infty}.
    \]
    To bound the norm of the gradient $\nabla v^\varepsilon$ we use the flow interchange technique introduced in \cite{matthes2009family}. In particular, we consider the measure $\calS^s\rho_{\tau, k}^\varepsilon$ as a competitor of $\rho_{\tau, k}^\varepsilon$. Let us denote the evolution of the free energy $\calF_\varepsilon$ along the heat flow by 
    \[
    \rmD_\calE \calF_\varepsilon(\rho) := \limsup_{s\downarrow 0} \frac{\calF_\varepsilon(\rho) - \calF_\varepsilon(S^s\rho)}{s} = \limsup_{s\downarrow 0} \int_0^1 -\frac{d}{dz}\Big|_{z=ts}\calF_\varepsilon(\calS^{z}\rho)dt.
    \]
    Since $W \in C^2$, the integration by parts gives the following bound on $\rmD_\calE \calF_\varepsilon$, where the term corresponding to the fixed interaction kernel $W$ is independent of $\rho$: %
    \[
    \begin{aligned}
     \MoveEqLeft \rmD_\calE \calF_\varepsilon(\rho) \geq  \liminf_{s\downarrow 0} \int_0^1 -\frac{1}{2}\frac{d}{dz}\Big|_{z=ts}\left(\iint W(x, y) (\calS^{z}\rho)(x) (\calS^{z}\rho)(y)d\sigma(x)d\sigma(y)\right)dt \\
     &\qquad + \limsup_{s\downarrow 0} \int_0^1 -\frac{1}{2}\frac{d}{dz}\Big|_{z=ts}\left(\iint V_\varepsilon(x, y) (\calS^{z}\rho)(x) (\calS^{z}\rho)(y)d\sigma(x)d\sigma(y)\right)dt\\
     &= \liminf_{s\downarrow 0} \int_0^1 \left(\int g_x \left( \nabla (W *(\calS^{z}\rho))(x), \nabla (\calS^{z}\rho) (x))\right)d\sigma(x)\right)\Big|_{z=ts}dt \\
     &\qquad + \limsup_{s\downarrow 0} \int_0^1 \left(\int g_x \left( \nabla (V_\varepsilon *(\calS^{z}\rho))(x), \nabla (\calS^{z}\rho) (x))\right)d\sigma(x)\right)\Big|_{z=ts}dt \\
     &\geq - \|\Delta W\|_{L^\infty} \int_0^1 \left(\iint(\calS^{z}\rho)(x) (\calS^{z}\rho)(y)d\sigma(x)d\sigma(y)\right)\Big|_{z=ts}dt \\
    & \qquad+ \limsup_{s\downarrow 0} -\int_0^1 \int (V_\e *(\calS^{z}\rho))(x) \Delta (\calS^{z}\rho)(x) d\sigma(x)\Big|_{z=ts}dt\\
     &\geq 
    \limsup_{s\downarrow 0} -\int_0^1 \int (V_\e *(\calS^{z}\rho))(x) \Delta (\calS^{z}\rho)(x) d\sigma(x)\Big|_{z=ts}dt
     - \|\Delta W\|_{L^\infty}.
    \end{aligned}
    \] 
    And since $\calS^{z}\rho \in C^\infty$ for arbitrary $\rho$, application of Corollary \ref{cor:flow-interchange} gives the following inequality:
    \begin{equation}
    \label{eq:v-1}
    \rmD_\calE \calF_\varepsilon(\rho) 
    \geq 
    \limsup_{s\downarrow 0} \int_0^1 \left\|\nabla \left(\sqrt[*]{V_\varepsilon} *(\calS^{z}\rho)\right)\right\|^2_{L^2(T\bbS^{n-1})}\Big|_{z=ts}dt - \|\Delta W\|_{L^\infty}.
    \end{equation}
    By construction, $\rho_{\tau, k}^\varepsilon$ satisfies
    \[
    \frac{1}{2\tau}W_2^2(\rho_{\tau, k}^\varepsilon, \rho_{\tau, k-1}^\varepsilon) + \calF_\varepsilon(\rho_{\tau, k}^\varepsilon) \leq \frac{1}{2\tau}W_2^2(\calS^s\rho_{\tau, k}^\varepsilon, \rho_{\tau, k-1}^\varepsilon) + \calF_\varepsilon(\calS^s\rho_{\tau, k}^\varepsilon).
    \]
    After rearranging, multiplying by $\tau$ and dividing by $s$ we obtain
    \[
    \tau \frac{\calF_\varepsilon(\rho_{\tau, k}^\varepsilon) - \calF_\varepsilon(S^s\rho_{\tau, k}^\varepsilon)}{s} \leq \frac1{2s}\left(W_2^2(\calS^s\rho_{\tau, k}^\varepsilon, \rho_{\tau, k-1}^\varepsilon) - W_2^2(\rho_{\tau, k}^\varepsilon, \rho_{\tau, k-1}^\varepsilon) \right),
    \]
    and after passing $s\to 0$ by definition of $\rmD_\calE \calF_\varepsilon$ we get
    \[
    \tau \rmD_\calE \calF_\varepsilon(\rho^\varepsilon_{\tau, k}) \leq \frac12\frac{d^+}{ds}\left(W_2^2(\calS^s\rho_{\tau, k}^\varepsilon, \rho_{\tau, k-1}^\varepsilon)\right)\Big|_{s= 0}.
    \]
    Assuming that $\calE(\rho^\varepsilon_{\tau, k-1}) < \infty$, the heat flow satisfies the EVI~\eqref{eq:EVI}, and after taking the $\limsup$ as $s\downarrow 0$ we thus obtain
    \begin{align}
    \tau \rmD_\calE \calF_\varepsilon(\rho^\varepsilon_{\tau, k}) &\leq \calE(\rho^\varepsilon_{\tau, k-1}) - \liminf_{s\downarrow 0}\calE(\calS^s\rho^\varepsilon_{\tau, k}) -\liminf_{s\downarrow 0}\frac{n-2}{2}W_2^2(\calS^s\rho^\varepsilon_{\tau, k}, \rho^\varepsilon_{\tau, k-1}) \notag \\
    &\leq \calE(\rho^\varepsilon_{\tau, k-1}) - \calE(\rho^\varepsilon_{\tau, k}). \label{eq:v-2}
    \end{align}
    Combining inequalities \eqref{eq:v-1} and \eqref{eq:v-2} we conclude that 
    \[
    \begin{aligned}
    \MoveEqLeft\tau \limsup_{s\downarrow 0} \int_0^1 \left\|\nabla \left(\sqrt[*]{V_\varepsilon} *(\calS^{z}\rho^\varepsilon_{\tau, k})\right)\right\|^2_{L^2(\bbS^{n-1})}\Big|_{z=ts}dt \\
    &\leq \calE(\rho^\varepsilon_{\tau, k-1}) - \calE(\rho^\varepsilon_{\tau, k}) + \tau\|\Delta W\|_{L^\infty}.
    \end{aligned}
    \]
    In particular, this inequality shows that $\calE(\rho^\varepsilon_{\tau, k-1}) < \infty$ implies $\calE(\rho^\varepsilon_{\tau, k}) < \infty$; since $\rho^\varepsilon_{\tau,0} = \rho_0$ satisfies $\calE(\rho_0)<\infty$, it follows that $\calE(\rho^\varepsilon_{\tau, k}) < \infty$ for all $k$.

    Applying Lemma \ref{lemma:heat} and using the dominated convergence theorem we pass to the limit $\calS^{z}\rho \to \rho$ as $z\to 0$, and using the lower-semicontinuity of the $H^1$-seminorm under $L^2$-convergence we conclude that
    \[
    \tau \left\|\nabla \left(\sqrt[*]{V_\varepsilon} *\rho^\varepsilon_{\tau, k}\right)\right\|^2_{L^2(T\bbS^{n-1})} \leq \calE(\rho^\varepsilon_{\tau, k-1}) - \calE(\rho^\varepsilon_{\tau, k}) + \tau\|\Delta W\|_{L^\infty}.
    \]
    Note that $\bar \rho = 1$ is the unique minimizer of the entropy on $\bbS^{n-1}$ and thus the entropy is bounded from below by $\calE_{\min} = \int \bar\rho \log \bar\rho \, d\sigma = 0$. As a result, summing the inequality above over $k$ we conclude that
    \[
    \|\nabla v^\varepsilon_{\tau, k}\|^2_{L^2(0,T;L^2(T\bbS^{n-1}))} = \int_{t=0}^T \left\|\nabla \left(\sqrt[*]{V_\varepsilon} *\rho^\varepsilon_{\tau, k}(t)\right)\right\|^2_{L^2(T\bbS^{n-1})} dt \leq \calE(\rho_0)  + T\|\Delta W\|_{L^\infty}.
    \]
    Since $v^\varepsilon_{\tau_k} \to \tilde v^\varepsilon$ weakly in $L^2(0, T;  L^2(\bbS^{n-1}))$ and the $H^1$-seminorm also is lower-semicontinuous under weak $L^2$-convergence, the norm of the limiting curve is bounded uniformly in $\varepsilon$, namely: 
    \[
    \|\nabla \tilde v^\varepsilon\|^2_{L^2(0, T;  H^1(\bbS^{n-1}))} < \calE(\rho_0)  + T\|\Delta W\|_{L^\infty}.
    \]
    Thus, the family $\{\tilde v^\varepsilon\}_{\e>0}$ is bounded in $L^2(0, T;  H^1(\bbS^{n-1}))$ and, by the Banach-Alaoglu theorem, therefore weakly relatively compact.
\end{proof}
\begin{lemma}[Convergence of $\{v^\varepsilon\}$]
\label{lem:comp-v}
    Let $\{\rho^{\varepsilon_\ell}\}_{\ell\in\bbN}$ be the weakly convergent sequence from Lemma \ref{lem:comp-rho}. Then for any $T>0$, the corresponding sequence of curves $(v^{\varepsilon_\ell})_{k\in\bbN}$ converges strongly in  $L^2(0, T;  L^2(\bbS^{n-1}))$ to the curve $\tilde v$ given by Lemma~\ref{lemma:bound-v}.
\end{lemma}
\begin{proof}
   The proof follows the steps of \cite[Proposition 4.3]{burger2023porous}. In particular, applying Proposition \ref{prop:savare} to the family $v^\varepsilon$ with the functional 
   \[
   \calF(v) :=\begin{cases} \|v\|^2_{H^1(\bbS^{n-1})}, \quad &v \in \calP_{ac}(\bbS^{n-1})\cap H^1(\bbS^{n-1}), \\
   +\infty, &\text{otherwise}.     
   \end{cases} 
   \] 
   and the distance $g(u, v) = W_1(u, v)$, in view of Lemma \ref{lem:was-bound} we get the result. We remark that due to the compactness of the sphere, the steps 1 and 2 are significantly simpler than in the proof of \cite[Proposition 4.3]{burger2023porous}. In fact, compactness of the sublevel sets of $\calF$ follows directly from the Rellich theorem, see for example \cite[Proposition 4.4]{taylor1996partial}, and tightness of the family $v_\varepsilon$ is a direct consequence of the uniform bound obtained in Lemma \ref{lemma:bound-v}.
\end{proof}

\subsection{Proof of Theorem \ref{th:main-diff}}
\label{sec:proof-main}

We now give the proof of Theorem~\ref{th:main-diff}, and in this section we will therefore adopt the assumptions of Theorem~\ref{th:main-diff} on $V_\e$ and $W$.

Using the definition of weak solutions of~\eqref{eq:intro-main} and the definition of the smoothed curve~$v^\varepsilon$, we conclude that the pair $(\rho^\varepsilon, v^\varepsilon)$ satisfies the following relation for every $\varepsilon \in \bbR_+$,  $\varphi \in C^2(\bbS^{n-1})$ and $t\geq0$:
\[
\begin{aligned}
    \MoveEqLeft\int_{\bbS^{n-1}} \varphi(x) d \rho_t^{\varepsilon}(x)-\int_{\bbS^{n-1}} \varphi(x) d \rho_0(x)\\
&=- \int_0^t \int_{\bbS^{n-1} \times \bbS^{n-1}}g_x\Big(\nabla_x \varphi(x),  \nabla_x W(x, y)\Big) d \rho_s^{\varepsilon}(y) d \rho_s^{\varepsilon}(x) d s\\
&\qquad-\int_0^t \int_{\bbS^{n-1} \times \bbS^{n-1}}g_x\Big(\nabla_x \varphi(x),  \nabla_x V_\varepsilon(x, y)\Big) d \rho_s^{\varepsilon}(y) d \rho_s^{\varepsilon}(x) d s \\
&\leftstackrel{L.\,\ref{lemma:z}}= -\int_0^t \int_{\bbS^{n-1} \times \bbS^{n-1}}g_x\Big(\nabla_x \varphi(x),  \nabla_x W(x, y)\Big) d \rho_s^{\varepsilon}(y) d \rho_s^{\varepsilon}(x) d s\\
&\qquad-\int_0^t \int_{\bbS^{n-1} }g_x\Big(\nabla_x \varphi(x),  \nabla_x v^\varepsilon_s(x)\Big)  v_s^{\varepsilon}(x) d\sigma(x)d s  \\
&\qquad+ \int_0^t\int_{\bbS^{n-1}} g_x(r_s^\varepsilon(x), \nabla v^\varepsilon_s(x))d\sigma(x)ds,
\end{aligned}
\]
where 
\begin{equation}
    \label{eqdef:residual}
r_s^\varepsilon(x) := \int \sqrt[*]{V_\varepsilon} (z, x)(\rho_s^{\varepsilon}(z)\Pi_{xz}\nabla_z \varphi(z))d\sigma(z) - (\sqrt[*]{V_\varepsilon} *\rho_s^{\varepsilon}) \nabla_x \varphi(x)
\end{equation}
is a residual term that follows from Lemma \ref{lemma:z} below. Comparing this expression with~\eqref{eq:solADE}, we observe that the proof of Theorem \ref{th:main-diff} thus relies on two facts: convergence of the residual $r^\varepsilon$ to zero, which we prove in Lemma~\ref{lem:residual},  and the equality of the limits $\rho^\varepsilon$ and $v^\varepsilon$, which we prove in Lemma \ref{lem:limits}. 

\medskip
We first show the missing step in the calculation above.
\begin{lemma}
\label{lemma:z}
The following equality holds for arbitrary $\varepsilon \in \bbR_+$:
\[
\begin{aligned}
\MoveEqLeft \int_0^t \int_{\bbS^{n-1} \times \bbS^{n-1}}g_x\Big(\nabla_x \varphi(x),  \nabla_x V_\varepsilon(x, y)\Big) d \rho_s^{\varepsilon}(y) d \rho_s^{\varepsilon}(x) d s \\
&=\int_0^t \int_{\bbS^{n-1} }g_x\Big(\nabla_x \varphi(x),  \nabla_x v^\varepsilon_s(x)\Big)  d v_s^{\varepsilon}(x) d s  \\
&\qquad - \int_0^t\int_{\bbS^{n-1}} g_x(r_s^\varepsilon(x), \nabla v^\varepsilon_s(x))d\sigma(x)ds,
\end{aligned}
\]
where $r_t^\e$ is given in~\eqref{eqdef:residual}.
\end{lemma}

\begin{proof}
Using Lemma \ref{prop:grad} we obtain
\[
    \begin{aligned}
       \MoveEqLeft \iint g_x\Big(\nabla_x \varphi(x),  \nabla_x V_\varepsilon(x, y)\Big) \rho_s^{\varepsilon}(y) \rho_s^{\varepsilon}(x)d\sigma(x) d\sigma(y)\\
       &=\iiint g_x\Big(\nabla_x \varphi(x),  \nabla_x \sqrt[*]{V_\varepsilon} (x, z) \cdot \sqrt[*]{V_\varepsilon}(z, y)  \rho_s^{\varepsilon}(y)\Big) \rho_s^{\varepsilon}(x)(d\sigma)^{3} \\
       &=\iint g_x\Big(\rho_s^{\varepsilon}(x)\nabla_x \varphi(x),  \nabla_x \sqrt[*]{V_\varepsilon} (x, z) \cdot v_s^\varepsilon(z)\Big) d\sigma(x)d\sigma(z) \\
       &=\int g_x\Big( \rho_s^{\varepsilon}(x)\nabla_x \varphi(x),  \sqrt[*]{V_\varepsilon} (x, z) * \Pi_{xz}\nabla_z v_s^\varepsilon(z)\Big)d\sigma(x)\\
       &=\iint g_x\Big( \sqrt[*]{V_\varepsilon} (x, z) \cdot (\rho_s^{\varepsilon}(x)\nabla_x \varphi(x)),  \Pi_{xz}\nabla_z v_s^{\varepsilon}(z)\Big)d\sigma(x)d\sigma(z) \\
       &=\iint g_z\Big( \sqrt[*]{V_\varepsilon} (x, z) \cdot (\rho_s^{\varepsilon}(x)\Pi_{zx}\nabla_x \varphi(x)),  \nabla_z v_s^{\varepsilon}(z)\Big) d\sigma(x)d\sigma(z).
    \end{aligned}
    \]
    Integrating over $s$ yields the result.
\end{proof}

\begin{lemma}
\label{lem:residual-pw}
For any $\varphi \in C^2(\bbS^{n-1})$, the residual term $r^\varepsilon$ satisfies for all $s\geq0$:
\[
\int \|r_s^\varepsilon(x)\|_{g_x}d\sigma(x) \to 0.
\]
\end{lemma}
\begin{proof}
Since $\varphi \in C^2$, the gradient $\nabla \varphi$ is Lipchitz continuous, meaning that there exists $L>0$ such that  $\|\Pi_{xz}\nabla_z \varphi(z) - \nabla_x \varphi(x)\|_{g_x} \leq L\dist(x, z)$. By Assumption \ref{ass:conv-root}, the square root $\sqrt[*]{V_\varepsilon} $ is non-negative and hence we obtain
  \[
\begin{aligned}
   \MoveEqLeft\int \|r_t^\varepsilon(x)\|_{g_x}d\sigma(x) = \int \left\| \int \sqrt[*]{V_\varepsilon} (z, x)(\rho_s^{\varepsilon}(z)\Pi_{xz}\nabla_z \varphi(z)) - (\sqrt[*]{V_\varepsilon} *\rho_s^{\varepsilon}) \nabla_x \varphi(x)\right\|_{g_x} d\sigma(x)\\
   &\leq \iint \sqrt[*]{V_\varepsilon} (z, x)\rho_s^{\varepsilon}(z)\left\|\Pi_{xz}\nabla_z \varphi(z) - \nabla_x \varphi(x)\right\|_{g_x} d\sigma(z)d\sigma(x) \\
   &\leq L\int \rho_s^{\varepsilon}(z)d\sigma (z) \int \sqrt[*]{V_\varepsilon}(z, x) \dist(z, x) d\sigma(x). 
\end{aligned}
\] 
We now fix any $z_0\in \bbS^{n-1}$ and calculate 
\begin{align*}
   \MoveEqLeft \lim_{\varepsilon \to 0}\int \|r_t^\varepsilon(x)\|_{g_x}d\sigma(x) \leq \lim_{\varepsilon \to 0}L\int \rho_s^{\varepsilon}(z)d\sigma (z) \int \sqrt[*]{V_\varepsilon}(z, x) \dist(z, x) d\sigma(x) \\
   &\leftstackrel{(*)}= L \lim_{\varepsilon \to 0} \int \sqrt[*]{V_\varepsilon}(z_0, x) \dist(z_0, x) d\sigma(x) \\
   &= 0,
\end{align*}
where the identity~$(*)$ follows from the rotational symmetry of $\sigma$, $\dist$, and $\sqrt[*]{V_\e}$, and where the final step follows from Lemma~\ref{lemma:uniform}.
\end{proof}

\begin{lemma}[$r^\varepsilon$ converges strongly to zero] 
\label{lem:residual}
The residual $r^\e$ in~\eqref{eqdef:residual} satisfies 
\[
\|r^\varepsilon\|_{L^2(0,T; L^2(T\bbS^{n-1}))} \to 0 \quad \text{as } \varepsilon \to 0.
\]
\end{lemma}
\begin{proof}
    As in \cite[Lemma 5.2, Corollary 5.1]{burger2023porous}, we combine Lemmas \ref{lemma:bound-v} and \ref{lem:residual-pw} with the Sobolev embedding theorem to get the result.
\end{proof}

\begin{lemma}[$\lim \rho^\varepsilon = \lim v^\varepsilon$]
\label{lem:limits} 
Let $\rho^{\varepsilon_k}$ be the weakly convergent sequence of curves and $v^{\varepsilon_k}$ be the sequence of corresponding smoothed curves. Let $\tilde \rho$ be the narrow limit of~$\rho^{\varepsilon_k}$ and $\tilde v$ be the weak $L^2_{\mathrm{loc}}(0, T;H^1(\bbS^{n-1}))$ limit of $v^{\varepsilon_k}$, then $\tilde \rho=\tilde v$.
\end{lemma}
\begin{proof}
    Fix $\varphi\in C_c([0,\infty){\times}\bbS^{n-1}$.  By definition of $v^\varepsilon$ we have for fixed $t\geq0$ 
    \[
    \begin{aligned}
    \int_{\bbS^{n-1}}  \varphi(t,x) dv_t^{\varepsilon_k}(x) &
    = \iint \varphi(t,x) \left(\sqrt[*]{V_{\varepsilon_k}}(x, y)\rho_t^{\varepsilon_k}(y)\right)d\sigma(x)d\sigma(y)\\
    &= \iint \left(\varphi(t,x) \sqrt[*]{V_{\varepsilon_k}}(x, y)\right)d\rho_t^{\varepsilon_k}(y)d\sigma(x) \\
    &= \int (\sqrt[*]{V_{\varepsilon_k}}{*}\varphi(t,\cdot))(y) d\rho_t^{\varepsilon_k}(y).
    \end{aligned}
    \]
    Since $\varphi$ is bounded, the same holds for the convolution. In addition, by Lemma~\ref{lemma:uniform} we have $\bra[\big]{\sqrt[*]{V_{\varepsilon_k}}{*}\varphi(t,\cdot)} \to \varphi$ uniformly on $\bbS^{n-1}$ for every $t\geq0$, thus
    \[
    \int \bra[\Big]{\varphi - \sqrt[*]{V_{\varepsilon_k}}*\varphi}d\rho_t^{\varepsilon_k} \to 0.
    \]
    Note that for any fixed $\varphi \in C_b$ the integral above is bounded uniformly in $t$, namely 
\[
\left|\int (\varphi - \sqrt[*]{V_{\varepsilon_k}}*\varphi)d\rho_t^{\varepsilon_k}\right| \leq  \left\|\varphi - \sqrt[*]{V_{\varepsilon_k}}*\varphi\right\|_{L^\infty} %
\]
and thus, applying the dominated convergence theorem, we conclude that
\[
\int_0^T dt \int_{\bbS^{n-1}} \varphi(x) d(\tilde\rho_t - \tilde v_t)(x) =\int_0^T dt \int_{\bbS^{n-1}} \lim_{k\to \infty}(\varphi - \sqrt[*]{V_{\varepsilon_k}}*\varphi)d \rho^{\varepsilon_k}_t(x) \to 0,
\]
which completes the proof.
\end{proof}

We are now ready to prove the main theorem.
\begin{proof}[Proof of Theorem \ref{th:main-diff}]
    Again we fix a time $T>0$.
Let $\tilde \rho = \lim \rho^{\varepsilon_k}$ and $\tilde v = \lim v^{\varepsilon_k}$ as above. Since $\rho^\varepsilon$ is a weak solution of \eqref{eq:intro-main}, the pair $(\rho^\varepsilon, v^\varepsilon)$ satisfies
\begin{align}
    \MoveEqLeft\int_{\bbS^{n-1}} \varphi(x) d \rho_t^{\varepsilon}(x)-\int_{\bbS^{n-1}} \varphi(x) d \rho_0(x)\notag \\
&= -\int_0^t \int_{\bbS^{n-1} \times \bbS^{n-1}}g_x\Big(\nabla_x \varphi(x),  \nabla_x W(x, y)\Big) d \rho_s^{\varepsilon}(y) d \rho_s^{\varepsilon}(x) d s\notag \\
&\qquad-\int_0^t \int_{\bbS^{n-1} }g_x\Big(\nabla_x \varphi(x),  \nabla_x v^\varepsilon_s(x)\Big)  v_s^{\varepsilon}(x) d\sigma(x)d s \notag  \\
&\qquad+ \int_0^t\int_{\bbS^{n-1}} g_x(r_s^\varepsilon(x), \nabla v^\varepsilon_s(x))d\sigma(x)ds.
\label{eq:th:main-diff:weak-sol}
\end{align}
Using the uniform bound on $\| \nabla v^\varepsilon\|_{L^2(0, T;  L^2(T\bbS^{n-1}))}$ from Lemma \ref{lemma:bound-v} and the convergence of the residual term proved in Lemma \ref{lem:residual} we get
\[
\begin{aligned}
\MoveEqLeft\int_0^t\int_{\bbS^{n-1}} g_x(r_s^\varepsilon(x), \nabla v^\varepsilon_s(x))d\sigma(x)ds \\
&\leq \|r^\varepsilon\|_{L^2(0, T;  L^2(T\bbS^{n-1}))}\| \nabla v^\varepsilon\|_{L^2(0, T;  L^2(T\bbS^{n-1}))} \to 0.
\end{aligned}
\]
Next, note that for arbitrary $\varphi \in C^\infty$, the function $g_x\Big(\nabla_x \varphi(x),  \nabla_x W(x, y)\Big)$ is uniformly bounded on $\bbS^{n-1}$. From the weak convergence $\rho^\varepsilon_s \stackrel{w}{\to} \tilde\rho_s$ we deduce 
\[
\begin{aligned}
\MoveEqLeft\int_{\bbS^{n-1} \times \bbS^{n-1}}g_x\Big(\nabla_x \varphi(x),  \nabla_x W(x, y)\Big) d \rho_s^{\varepsilon}(y) d \rho_s^{\varepsilon}(x) \\
&\to \int_{\bbS^{n-1} \times \bbS^{n-1}}g_x\Big(\nabla_x \varphi(x),  \nabla_x W(x, y)\Big) d \tilde\rho_s(y) d \tilde\rho_s(x). 
\end{aligned}
\]
Thus, the dominated convergence theorem guarantees the convergence of the first term in~\eqref{eq:th:main-diff:weak-sol}.
Finally, by Lemma \ref{lemma:bound-v}, the sequence $v^{\varepsilon_k}$ satisfies
\[
v^{\varepsilon_k} \stackrel{w}{\to} \tilde v, \qquad \text{in } L^2(0, T;  H^1(\bbS^{n-1})), 
\]
along a subsequence and, by Lemma \ref{lem:comp-v}
\[
v^{\varepsilon_k} \to \tilde v, \qquad \text{strongly in } L^2(0, T;  L^2(\bbS^{n-1})). 
\]
As a result, for any $\varphi \in C^\infty$, by the Cauchy-Schwartz inequality we obtain
\[
\begin{aligned}
\MoveEqLeft\left| \int_0^t \int_{\bbS^{n-1} }g_x\Big(\nabla_x \varphi(x),  \nabla_x v^\varepsilon_s(x)\Big)  d v_s^{\varepsilon}(x) d s   \right|\\
&\leq \left| \int_0^t \int_{\bbS^{n-1} }g_x\Big(\nabla_x \varphi(x),  \nabla_x v^\varepsilon_s(x)\Big)(v_s^{\varepsilon} - \tilde v_s) d\sigma d s   \right| \\
&\qquad+ \left| \int_0^t \int_{\bbS^{n-1} }g_x\Big(\nabla_x \varphi(x),  (\nabla_x v^\varepsilon_s(x) - \nabla_x \tilde v_s(x)\Big)\tilde v_s  d\sigma d s   \right| \\
&\leq \|\nabla \varphi\|_{L^\infty(T\bbS^{n-1})}\|\nabla v^\varepsilon_s\|_{L^2(0, T;  L^2(T\bbS^{n-1}))} \|v_s^{\varepsilon} - \tilde v_s\|_{L^2(0, T;  L^2(T\bbS^{n-1}))} \\
&\qquad+   \left|\int_0^t \int_{\bbS^{n-1} }g_x\Big(\tilde v_s\nabla_x \varphi(x),  (\nabla_x v^\varepsilon_s(x) - \nabla_x \tilde v_s(x)\Big)d\sigma d s \right| \to 0,
\end{aligned}
\]
since $\tilde v\nabla_x \varphi \in L^2(0, T;  L^2(T\bbS^{n-1}))$. Combining the results, we conclude that the pair $(\tilde\rho, \tilde v)$ satisfies 
\[
\begin{aligned}
    \MoveEqLeft\int_{\bbS^{n-1}} \varphi(x) d \tilde\rho_t(x)-\int_{\bbS^{n-1}} \varphi(x) d \rho_0(x)\\
&= -\int_0^t \int_{\bbS^{n-1} \times \bbS^{n-1}}g_x\Big(\nabla_x \varphi(x),  \nabla_x W(x, y)\Big) d \tilde \rho_s(y) d \tilde \rho_s(x) d s\\
&\qquad-\int_0^t \int_{\bbS^{n-1} }g_x\Big(\nabla_x \varphi(x),  \nabla_x \tilde v_s(x)\Big)  d \tilde v_s(x) d s,
\end{aligned}
\]
Using Lemma \ref{lem:limits} we deduce that $\tilde \rho = \tilde v$. Moreover, arguing analogously to the proof of Lemma \ref{lem:limits} we conclude that
\[
\begin{aligned}
\MoveEqLeft
\int_0^t \int_{\bbS^{n-1} \times \bbS^{n-1}}g_x\Big(\nabla_x \varphi(x),  \nabla_x W(x, y)\Big) d \tilde \rho_s(y) d \tilde \rho_s(x) d s\\
&= \int_0^t \int_{\bbS^{n-1} \times \bbS^{n-1}}g_x\Big(\nabla_x \varphi(x),  \nabla_x W(x, y)\Big)  \tilde v_s(y)  \tilde v_s(x) d\sigma(x)d\sigma(y)d s,
\end{aligned}
\]
and therefore the curve $\tilde v$ is a weak solution of \eqref{eq:intro-diffuse} in the sense of Definition~\ref{def:sol-ade}.
\end{proof}

\section{On the relation to transformer models}
\label{sec:relevance}

\subsection{Transformers}
In this section we present a toy transformer model with two self-attention heads as a motivating example for our analysis. We argue that the choice of the model with local repulsion and global attraction is well-motivated from the application perspective of transformers in natural language processing. We also interpret the boundedness of the solutions from the machine-learning perspective and claim that it is a desirable behaviour for the given model.

Transformers are a class of machine-learning models primarily designed for natural language processing tasks. A common approach in natural language processing is to build a vocabulary consisting of all possible words (or other small lexical elements called \emph{tokens}) and assign a (unit) vector value to every element of the vocabulary. Having done so, every text can than be split into a sequence of words and represented as a sequence of vectors corresponding to the given tokens. In particular, a sentence of length $d$ has a representation  $(x_i)_{1\leq i\leq d}$, $x_i \in \bbR^{n}$, where $n$ is the dimension of the model. 

 A transformer model operates on such representations and consists of \emph{self-attention} blocks, which have been first introduced by Vaswani et al. in \cite{vaswani2017attention}, as well as linear and normalization layers. A \emph{self-attention} layer $SA: \bbR^{n \times d} \to \bbR^{n \times d}$ maps a sequence of $d$  vectors in $\R^n$ into a similar sequence of vectors of the same size  and has the structure
\begin{equation}\label{transformer}
  SA(X)_i := \frac{1}{\sum_{j=1}^d e^{\left<Qx_i, Kx_j\right>}}\sum_{j=1}^d e^{\left<Qx_i, Kx_j\right>}Vx_j, \quad 1\leq i \leq d, \notag
\end{equation}
where $K, Q, V \in \bbR^{n\times n}$ are real-valued matrices. 
In this work we consider a simple version of a transformer, namely a residual network with only self-attention layers. In this model every vector $x_i$ follows the dynamics:
\begin{equation}
\label{eq:one-head}
  x_i^{k+1} = x_i^{k} + SA(X^k)_i =x_i^{k} + \frac{1}{\sum_{j=1}^d e^{\left<Q^kx_i^k, K^kx_j^k\right>}}\sum_{j=1}^d e^{\left<Q^kx_i^k, K^kx_j^k\right>}V^kx_j^k. 
\end{equation}
Note that this dynamics is different from the `training dynamics', corresponding to the evolution of parameters $Q^k$, $K^k$, and $V^k$ during optimization. In this case the index~$k$ corresponds to the $k$-th layer of the model but not to the $k$-th step of the training procedure. 

Both inputs and outputs of a transformer are sequences of (unit) vectors. The output sequence, however, does not have a direct interpretation and an additional model is always used to make a decision. The typical choice of the decision being made is the \emph{next token prediction}, the standard formulation of the language modeling problem, and we give an illustrative example in Section \ref{ssec:scaling}. In this example we also give a \emph{synthetic} interpretation of the output of a transformer in the absence of the additional model.

The model \eqref{eq:one-head} can be seen as a time-discretization of an interacting particle system and thus can also be studied on the level of measures as suggested in \cite{sander2022sinkformers}. In  \cite{geshkovski2024mathematical} it was proposed to further reduce the model in order to simplify the analysis. In particular, the following `toy transformer model' was introduced:
\begin{equation}
\label{eq:toy-transformer}
\dot x_i  = \frac{1}{d}P_{T_{x_i}\bbS^{n-1}}\bra[\bigg]{\sum_{j=1}^d e^{\beta\left<x_i, x_j\right>}x_j}, \qquad  \beta >0,
\end{equation}
as a proxy for \eqref{eq:one-head}. Here $P_{T_{x_i}\bbS^{n-1}}$ is the orthogonal projection in $\R^n$ onto the tangent plane at $x_i$. The system \eqref{eq:toy-transformer} corresponds to~\eqref{eq:one-head} with a specific choice of the parameters, namely $K^k =I,\ V^k=\alpha I, \ Q^k = \beta I$, with a few additional modifications; we refer the reader to \cite{geshkovski2024mathematical} for the details. In the measure-valued setting this `toy transformer model' is equivalent to the aggregation equation~\eqref{eq:intro-main} with $W_\beta(x, y) := -\frac{1}{\beta}e^{\beta\left<x, y\right>}$ and $V_\e=0$.

At the same time, real-world language models have more involved structure than~\eqref{eq:one-head}, and one of the key differences is that every residual step includes summation over \emph{several} self-attention heads. In particular, the residual step of a transformer with $M$ heads takes the form
\begin{equation*}
  x_i^{k+1} = x_i^{k} + \sum_{m=1}^M SA_m(X^k)_i =x_i^{k} + \sum_{m=1}^M \frac{1}{\sum_{j=1}^d e^{\left<Q_m^kx_i^k, K_m^kx_j^k\right>}}\sum_{j=1}^d e^{\left<Q_m^kx_i^k, K_m^kx_j^k\right>}V_m^kx_j^k, 
\end{equation*}
where $K_m^k, Q_m^k, V_m^k \in \bbR^{n\times n}$ are the parameters of $m$-th head of the $k$-th layer. Applying similar simplifications as in the single-head setting we obtain the continuous dynamics
\[
\dot x_i  = \frac{1}{d}\sum_{m=1}^M\alpha_m\sum_{j=1}^d e^{\beta_m\left<x_i, x_j\right>}x_j, 
\]
where $\beta_m$ is the interaction parameter of the $m$-th self-attention head and $\alpha_m$ is the weight of the corresponding head. The measure-valued counterpart in this case takes the form
\[
\partial_t\mu_t + \sum_{m=1}^M\nabla \cdot(\mu_t \nabla_x  W_{m}(x, \cdot) *\mu_t) = 0, \quad \text{where } W_{m} = \alpha_m W_{\beta_m}.
\]

In this work we consider a model with $M=2$ heads and we assume that the first head is \emph{globally attractive}, which corresponds to $\alpha_1, \beta_1 > 0$ and $\beta_1 \sim 1$ and the second head is \emph{locally repulsive}, corresponding to the parameters $\beta_2 \gg 1$ and $\alpha_2 < 0$. Note that $\beta_2 \gg 1$ implies that the interaction is strongly localized and $\alpha_2 < 0$ guarantees that it is repulsive. 
This leads to the family~\eqref{eq:intro-main} of evolution equations 
where the fixed interaction kernel~$W$ is the attractive self-attention head $W := \alpha_1 W_{\beta_1}$, and the localized kernel $V_\varepsilon$ is of the form $V_\varepsilon:= \alpha_\varepsilon W_{\beta_\varepsilon}$ with $\beta_\varepsilon = \varepsilon^{-1}$ and 
\[
\alpha_\varepsilon = \left(\int e^{\beta_\varepsilon\left<x_0, x\right>} d\sigma(x)\right)^{-1},
\qquad\text{for arbitrary $x_0 \in \bbS^{n-1}$. }
\]
We also remark that the fixed interaction kernel may include any finite number of self-attention heads with bounded parameters $\alpha_m, \beta_m < C$. The main question of this work is the behavior of the solutions in the limit of $\varepsilon \to 0$ and we discuss the relevance of such a setting below.

Note that the behaviour of transformers with attractive interaction, corresponding to $\alpha, \beta > 0$, is extensively studied in the range of recent works including \cite{geshkovski2024emergence, geshkovski2024mathematical, geshkovski2024dynamic, criscitiello2024synchronization, bruno2024emergence, bruno2025multiscale, polyanskiy2025synchronization, alcalde2025attention} and the repulsive interaction case is partially covered in \cite{geshkovski2024emergence, burger2025analysis, bruno2025multiscale, alcalde2025attention}. Nevertheless, in all of these papers the study is restricted to a single-head transformer model in which the sets of repulsive and attractive directions are disjoint. In other words, the tokens repel along some directions and attract along others. To the best of our knowledge, this work is the first theoretical analysis of toy transformer models with \emph{competing} attractive and repulsive forces in the sense that attraction and repulsion happen along the same direction.

\begin{remark}[Linear diffusion]
    The limit of a singular interaction kernel has also been considered in the presence of token-dependent rescaling, namely a prefactor of the form $\left(\sum_{j=1}^d e^{\left<Qx_i^k, Kx_j^k\right>}\right)^{-1}$, in \cite{sander2022sinkformers, bruno2025multiscale}. This model has been related to the heat equation. Formally, in the limit of the localized kernel, the inverse prefactor converges to the underlying measure, and thus the corresponding continuity equation takes the form
    \[
    \partial_t \mu_t - \nabla \cdot \left(\frac{d \mu_t}{d\mu_t }\nabla \mu_t\right) = \partial_t \mu_t -\Delta \mu_t = 0.
    \]
    We expect that the techniques used in this paper may also be of use to prove convergence of the solutions of rescaled transformers to the heat flow. We also remark that the aggregation model with the transformer interaction kernel in the presence of linear diffusion has been recently studied in \cite{shalova2024anoisy-transformer, balasubramanian2025structure}.
\end{remark}

\begin{remark}[Equivalence with the switching model]
    Consider the model with two \emph{alternating} self-attention heads, namely
    \begin{equation*}
  x_i^{k+1} = x_i^{k} +  \alpha SA_1(X^k)_i, \qquad  x_i^{k+2} = x_i^{k+1} +  \alpha SA_2(X^{k+1})_i.
\end{equation*}
In this case the model switches from head $SA_1$ and $SA_2$ and back at every iteration of the algorithm. For small $\alpha$ such a model can be interpreted as a splitting scheme applied to the ODE driven by the sum of the contribution of two heads
\[
\dot x_i  = \frac{1}{2d}\sum_{m=1}^2\sum_{j=1}^d e^{\beta_m\left<x_i, x_j\right>}x_j.
\]
Such splitting is a common approach in numerical solvers of various PDEs, including aggregation equations and the porous-medium PDE; see e.g.~\cite{holden2010splitting}.
\end{remark}

\subsection{Properties of the exponential kernel}
We consider the family of kernels of the form
\[
V_\varepsilon(x, y) : = \alpha_\varepsilon e^{\left<x, y\right>/\varepsilon}, \quad \text{where } \quad \alpha_\varepsilon= \left(\int e^{\left<x, y\right>/\varepsilon} d\sigma(x)\right)^{-1}.
\]
For every $\varepsilon \in \bbR_+$ the kernel $V_\varepsilon$ is a smooth function and hence $V_\varepsilon\in H^1(\bbS^{n-1} \times  \bbS^{n-1})\cap C_b(\bbS^{n-1} \times  \bbS^{n-1})$. It was calculated in~\cite[Proposition~6.1]{shalova2024anoisy-transformer} that the spherical harmonics decomposition of $V_\varepsilon$ has the form
\[
\hat V_{\varepsilon,l} = \alpha_\varepsilon C(n, \varepsilon)I_{l
+\frac{n-2}{2}}(1/\varepsilon).
\]
At the same time, the normalization constant is the projection of $f_\varepsilon= e^{\left<x_0, \cdot\right>/\varepsilon}$ onto the constant function, namely the spherical harmonic $Y_{0, 0}$, and thus  $\alpha_\varepsilon = C(n, \varepsilon)I_{\frac{n-2}{2}}(1/\varepsilon)$. Since the modified Bessel functions $I_z(\beta)$ are positive and decreasing in $z$, we conclude that $\hat V_{\varepsilon,l} \leq 1$. Moreover, for every fixed $l\in \bbN$ we have
\[
\frac{I_{l
+\frac{n-2}{2}}(1/\varepsilon)}{I_{\frac{n-2}{2}}(1/\varepsilon)} \sim \frac{e^{1/\varepsilon}(\sqrt{2\pi\varepsilon^{-1}})^{-1}(1+o(\varepsilon))}{e^{1/\varepsilon}(\sqrt{2\pi\varepsilon^{-1}})^{-1}(1+o(\varepsilon))} \to 1
\qquad\text{as }\e\to0.
\]
Finally, we need to verify that $\sum_{l} l^n\hat V_{\varepsilon, l} < \infty$. Since $f(x, y) = e^{\left<x, y\right>/\varepsilon}$ is a smooth function for every $\varepsilon >0$, we conclude that for every $p\in \bbN$ it holds that $\Delta^p f \in L^2(\bbS^{n-1}\times \bbS^{n-1})$. As a result we conclude that for every $p$ the following sum is finite
\[
\sum_{l} |\lambda_l|^p\hat V^2_{\varepsilon, l} = \left<f, \Delta^p f\right> \lesssim \sum_{l} l^{2p}\hat V^2_{\varepsilon, l} < \infty.
\]
Applying the Cauchy-Schwartz inequality we conclude that the exponential family of kernels satisfies the localization Assumption \ref{ass:repulsive}.

The key difficulty for establishing convergence of the transformer model is verification of Assumption \ref{ass:conv-root}, in particular the pointwise non-negativity. Alternatively, one could aim to work with a different distance function in Proposition \ref{prop:savare}.

\subsection{On the choice of the scaling} 
\label{ssec:scaling}
We argue that the setting of global attraction and local repulsion is optimal from the natural language processing perspective. Consider the problem of a missing (or next) token prediction: in this case the output distribution should be interpreted through the lens of possible semantics of the input text. In particular, in this context the global attraction force corresponds to the selection of a finite number of possible semantics and local repulsion provides a tool to ensure linguistic variability. In other words, the attractive kernel is responsible for the choice of \emph{meanings} and the local repulsion allows the model to choose among various \emph{synonyms} carrying the same \emph{meaning}. We clarify this remark on an example. 

Consider the following next token prediction task: we are given the sentence \[
\text{``A cat sat on a (?)'',}
\]
and are asked to predict the probability of the last word, denoted by (?). The possible answers might be \emph{mat}, \emph{couch}, \emph{sofa} or, maybe, \emph{tree}. While the answers \emph{couch} and \emph{tree} have very distinct semantic, the answers \emph{couch} and \emph{sofa} are semantically very similar. As a result, we expect the vector representations of \emph{couch} and \emph{sofa} to be almost identical, $\left<x_{couch}, x_{sofa}\right> \approx 1$, while the representations of \emph{couch} and \emph{tree} to be significantly distinct, $\left<x_{couch}, x_{tree}\right> < 1 - \delta$. 

As discussed in \cite{geshkovski2024mathematical}, clustering of the tokens can be interpreted as an extraction of a finite number of semantics. In particular, in \cite{geshkovski2024mathematical} it is shown that the solutions of the purely attractive model (under rather mild assumptions) converge to a single point, which can be interpreted as a choice of a single semantic (or even a single token). In addition, the attractive model has also been shown to exhibit metastable behaviour \cite{geshkovski2024dynamic, bruno2024emergence}, which shows that on a finite-time horizon the solution might concentrate on a finite number of different semantics. We argue that the local repulsion complements the picture of the global clustering by providing a tool to ensure local variability. 

In terms of the example, the global clustering would correspond to predicting one single option, for example \emph{couch}. A metastable state with two clusters would correspond to having a probability measure concentrated on the words \emph{couch} and \emph{tree}. At the same time, the local variability mechanism would smooth the bi-modal distribution and would allow for all close enough synonyms of both \emph{couch} and \emph{tree}.

\section{Points of discussion}
\label{sec:discussion}
\subsection{The fixed-$\varepsilon$ regime}
Note that the for fixed $\varepsilon >0$ the model corresponds to an aggregation PDE with interaction kernel $U_\varepsilon = W + V_\varepsilon$ with spherical harmonics decomposition $\hat U_{\varepsilon, k} = \hat W_{k} + \hat V_{\varepsilon, k}$. In particular, assuming that $-W$ is a stable kernel in the sense that $\hat W_{k} <0$ for all $k\in\bbN$, the addition of the repulsive kernel will lead to cancellation of the high harmonics. Here we assumed that the coefficients of the repulsive kernel $V_\varepsilon$ decay more slowly in absolute value than the coefficients of the attractive kernel $W$. 

Considering the same model in the presence of noise, the results from~\cite{shalova2024anoisy-transformer} imply that the model will only exhibit bifurcations corresponding to the low harmonics. At the same time, since the kernel is no longer guaranteed to be decreasing, the minimizers might correspond to non-synchronized measures. In particular, the minimizers might be multimodal in contrast to the pure aggregation case.

\subsection{Extensions}
We argue that our result can be generalized to a larger class of manifolds. The key observations allowing to establish the desired convergence are (a) the structure of the interaction kernel of form
\[
W(x, y) = \sum_l \hat W_l Y_l(x) Y_l(y),
\]
where $\{Y_l\}_{l\in\bbN}$ are the eigenfunctions of the Laplace-Beltrami operator,  and (b) the bound on the Wasserstein distance under the convolution as in Lemma~\ref{lem:was-bound}. We conjecture that the latter can be generalized to  more general smooth compact manifolds. 

We also remark that, for example, the heat kernel has the desired representation on an arbitrary smooth Riemannian manifold $\calM$. Formally, the heat kernel also converges to the point-estimation kernel on an arbitrary manifold and is thus a natural candidate to model the local repulsion on manifolds in the given context.

Finally, note that $\bbT^1 = \bbS^1$ and thus our analysis directly applies to the aggregation PDE on $\bbR$ with periodic boundary conditions.

\bibliographystyle{myalpha}
\bibliography{bibliography}

\appendix
\section{Differential forms}
\label{sec:geometry}

We recall a number of facts from differential geometry, in particular the geometry of Riemannian manifolds. Good background references are~\cite{Willmore65,Jost05,lee2018introduction}. 

\subsection{Generalities}

Let $(\calM, g)$ be a smooth Riemannian manifold (without boundary) with a metric $g$, and we assume that the reader is familiar with geodesics, connections, and the covariant derivative. 

\smallskip

In this work we will always consider the Levi-Civita connection.
Given this connection, for every point $x\in \calM$ and every tangent vector $v\in T_x\calM$ there exists a unique
geodesic $\gamma_{x,v}: [0,1] \to \calM$ with initial conditions $\gamma(0) = x, \ \gamma'(0) = v$. Then the exponential map is defined to be the end point of this geodesic:
\[
\exp_x(v) = \gamma_{x,v}(1).
\]
For small $v$ the exponential map is invertible, and we write the inverse as the `logarithmic map' $\log_x$. The derivative of the squared distance also is well-defined for short distances, and can be expressed in terms of the logarithmic map:
\begin{equation}
\label{eq:deriv-of-dsquared}
\nabla_y \dist^2(x,y) = -2\log_y x.
\end{equation}

For a smooth function $f: \calM \to \bbR$ its differential at a point $x\in\calM$ is a linear map $df_x: T_x\calM \to \bbR$ such that for any smooth curve satisfying $\gamma(0) = x, \ \gamma'(0) = v$ it holds that
\[
df_x(\gamma'(0)) = (f\circ \gamma)'(0),
\]
where the expression on the right hand side $f\circ \gamma$ is a curve in $\bbR$. 
The gradient of a smooth function $f: \calM \to \bbR$ is a vector field $\grad f$ which for any vector field $Z$ on $\calM$ and any point $x\in \calM$ satisfies
\[
g_x((\grad f)_x, Z_x) = df_x(Z_x).
\]

\begin{example}
    On the unit sphere $\calM = \bbS^{n-1}$ equipped with the distance $\dist(x, y) = \arccos(\left<x, y\right>)$ the manifold gradient $\grad_{\bbS^{n-1}} f$ in Euclidean coordinates is equal to the projection of the Euclidean gradient onto the tangent space at $x$:
    \[
    \grad_{\bbS^{n-1}} f_x = \nabla_{\bbR^n} f_x -  \left<\nabla_{\bbR^n} f_x, x\right> x,
    \]
    where $\left<\cdot, \cdot \right>$ is a Euclidean scalar product and $\nabla_{\bbR^n} f_x = \left(\frac{\partial f(x)}{\partial x_1}, \dots, \frac{\partial f(x)}{\partial x_n} \right)$.
\end{example}

The divergence of a smooth vector field $X$ on a manifold is the trace of the covariant derivative $\nabla X$ with Levi-Civita connection:
\[
\divr X := \tr (\nabla X),
\]
where $\nabla X$ is an object which for every smooth vector field $Y$ satisfies $\nabla X(Y) = \nabla_Y X$. 
In particular, if $\{e_i\}$ is a local orthonormal basis of the tangent bundle $T\calM$, then
\[
\divr X = \sum_i \left<\nabla_{e_i} X, e_i\right> = \sum_i g(\nabla_{e_i} X, e_i).
\]
An $n$-dimensional Riemannian manifold has a canonical volume measure $m$ which in local coordinates takes the form
\[
dm = \sqrt{\det g_{ij}}dx,
\]
where $g_{ij}$ is the metric tensor in local coordinates and $dx$ is the Lebesgue volume element in $\bbR^n$. As a result for any compact manifold without boundary $(\calM, g)$ we get the following rule of integration by parts:
\[
\int \phi \divr X \, dm = -\int g(\grad \phi, X )\,dm
\]
for any $\phi \in C^\infty(\calM)$.

The Laplace-Beltrami operator is a generalization of the Laplace operator to the manifold setting, namely for any smooth function $f: \calM \to \bbR$ such that $\grad f$ is a smooth vector field the action of the Laplace-Beltrami operator is defined as
\[
\Delta f := \divr (\grad f).
\]
\begin{example}[Corollary 1.4.3 {\cite{dai2013approximation}}]
    On a unit sphere $\calM = \bbS^{n-1}$ equipped with distance $\dist(x, y) = \arccos(\left<x, y\right>)$ the Laplace-Beltrami operator $\Delta f$ is equal to the Euclidean Laplacian of the function $\tilde f : \bbR^n \to \bbR$ defined as $\tilde f (x) = f(x/\|x\|)$:
    \[
    \Delta_{\bbS^{n-1}} f = \Delta_{\bbR^n} \tilde f, 
    \]
    where $\Delta_{\bbR^n} = \sum_i \frac{\partial^2}{\partial x_i^2}$.
\end{example}

\subsection{Parallel transport}
\label{app:parallel}
Consider a smooth curve $\gamma: [0,1] \to \calM$ and a connection on $\calM$. The parallel transport of a vector $v\in T_x \calM$ along $\gamma$ is a vector field $V$ on $\gamma$ satisfying the following properties:
\begin{itemize}
    \item $\nabla_{\gamma'(s)} V_{\gamma(s)} = 0$ for all $s\in (0,1)$,
    \item $V_{\gamma(0)} = v$.
\end{itemize}
For $0\leq s \leq t\leq 1$ the linear map $\Gamma(\gamma)_s^t: T_{\gamma(s)}\calM \to T_{\gamma(t)}\calM$ satisfying $\Gamma(\gamma)_s^t V_{\gamma(s)} := V_{\gamma(t)}$ for arbitrary $V_{\gamma(s)} \in T_{\gamma(s)}\calM$ is called the parallel transport map along $\gamma$.

Since in this work we consider the Levi-Civita connection, the parallel transport along any smooth curve is metric-preserving, in the sense that for any $u, v \in T_{\gamma(s)}\calM$ we have 
\[
g_{\gamma(s)}(u, v) = g_{\gamma(t) }\bra[\big]{\Gamma(\gamma)_s^t u , \Gamma(\gamma)_s^t v }.
\]
Applying this property to the geodesic curves we obtain the following characterization. For two points $x, y \in \calM$ such that there exists a unique geodesic $\gamma_{x\to y}$, let $v_{x\to y} = \log_x y$, then
\[
x = \exp_y  v_{y\to x},
\]
where $v_{y\to x} = - \Gamma(\gamma_{x\to y})_0^1 v_{x\to y} $ and $\|v_{y\to x}\|_{L^2(T_y\calM)} =\|v_{x\to y}\|_{L^2(T_x\calM)}$.

\section{Distance between geodesics on a sphere}
\label{app:geodesics}
\begin{proof}[Proof of Lemma \ref{lem:was-geo}] 
W.l.o.g. let $\|v_x\| = 1$; note that rescaling of $v_x$ is equivalent to the rescaling of time and thus does not change the character of the dynamics.

\emph{Step 1: Reducing the problem to $\bbS^{2}$.} We begin by showing that the problem can be reduced to the three-dimensional setting. For $n\leq 3$ it is trivially true. Assume that $n\geq 4$, then we argue as follows. Every geodesic $t\mapsto \exp_x t v_x$ forms a great circle which lies on the plane in $\bbR^n$ spanned by vectors $x$ and $v_x$. Thus, it is enough to show that the dimension of the $\text{span}\{x, y, v_x, v_y\}$ is at most $3$. W.l.o.g. assume that $x = (1, 0, 0, \dots 0)$ and $y = (\cos \theta, \sin \theta, 0, \dots 0)$. First note that if $x \parallel y$, the condition is trivially satisfied and thus it enough to consider $\theta \neq k\pi$. Moreover, in this case the parallel transport map is the rotation matrix of the form
    \[
    \Pi_{yx}=
    \begin{pmatrix}
        \cos \theta &  -\sin \theta & 0 \\
        \sin \theta & \cos \theta & 0 \\
        0& 0& I
    \end{pmatrix}.
    \]
As a result, the vector $v_y$ takes the form
\[
v_y = \Pi_{yx}v_x = \begin{pmatrix} v_x^1 \cos \theta - v_x^2\sin \theta\\ v_x^1 \sin \theta + v_x^2 \cos \theta \\
\vdots \\ v_x^n.
\end{pmatrix}
\]
Note that all the components of the vector $v_x - v_y$ except for the first two are zero, and thus we conclude that $v_x - v_y \in \text{span}\{x, y\}$ and thus $\dim \text{span}\{x, y, v_x, v_y\} \leq~3$.

\begin{figure}
    \centering
    \includegraphics[width=0.5\linewidth]{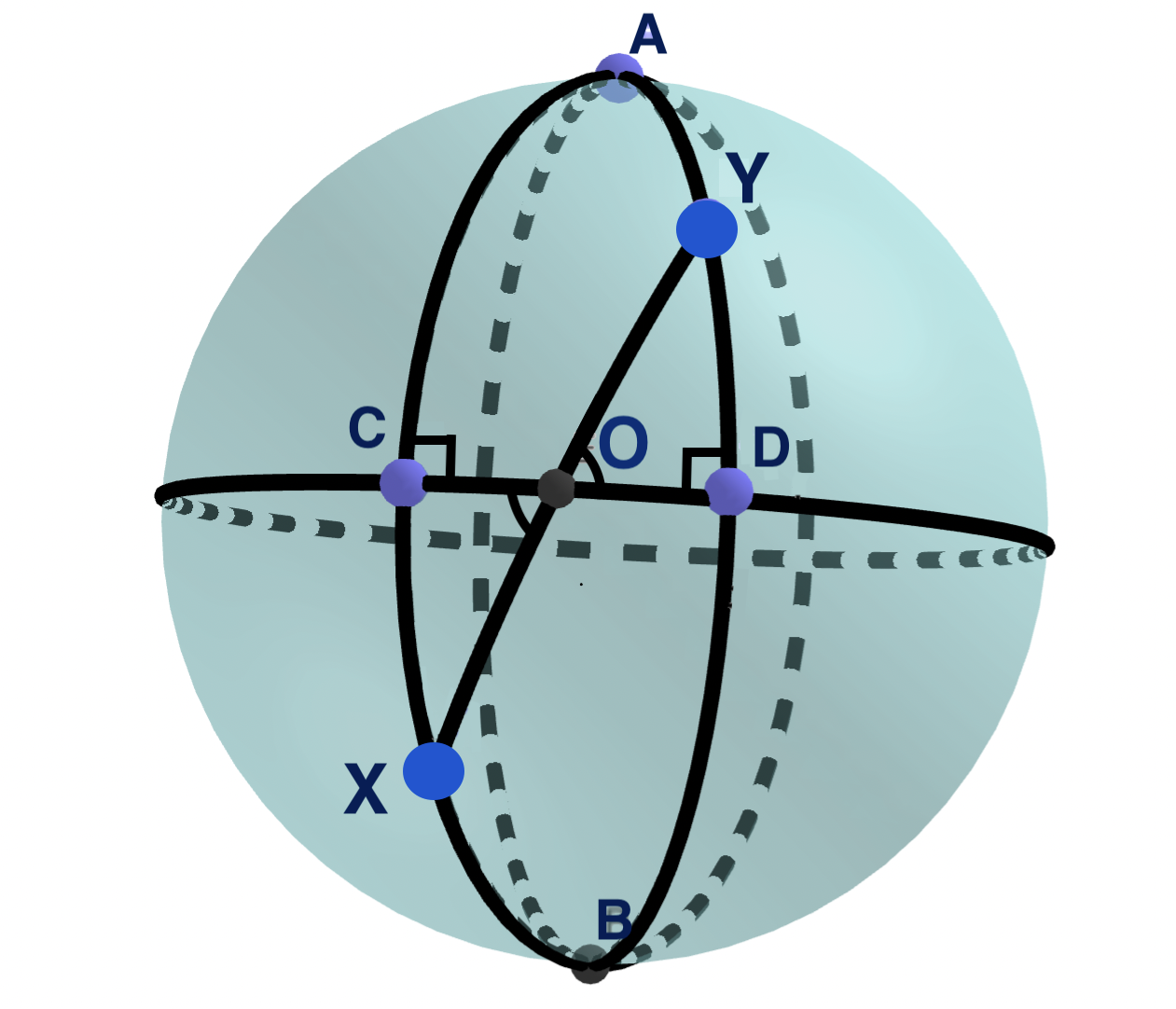}
    \caption{}
    \label{fig:sphere}
\end{figure}

\emph{Step 2: The static problem.}
Since the problem is intrinsically $3$-dimensional, we introduce the following construction on $\bbS^2$, see Figure \ref{fig:sphere}. Given two points $X, Y \in\bbS^{2}$ and the unit-length velocity vectors $v_x, v_y =\Pi_{yx}v_x$, we draw the corresponding geodesics. We call the points of intersections of the geodesics $A$ and $B$. By the metric preservation of the parallel transport map we conclude that $\angle AXY = \angle BYX$, where $\angle AXY$ denotes the angle of the spherical triangle, since
\[
\begin{aligned}
\|\log_xy\|\cos (\angle AXY )&= g_x(v_x, \log_xy) = g_y(\Pi_{yx}v_x, \Pi_{yx}\log_xy) \\
&= g_y(v_y, -\log_y x)= \cos (\angle BYX)\|\log_yx\|.
\end{aligned}
\]
At the same time, by construction $\angle YXB = \pi - \angle AXY = \pi - \angle BYX =\angle AYX$ and $\angle XAY = \angle XBY$. Since the triangles $AYX$ and $BXY$ share the side $XY$ and the correspondning angles are the same we conclude that the triangles are identical. This implies that $AY = BX$, where with a slight abuse of notation we use $AY = \dist(A, Y)$ etc. 

Let $C, D$ be the medians of both half-circles $AB$ (see Figure \ref{fig:sphere}) and draw a geodesic through $C$ and $D$. Let $O$ be the point of intersection of $CD$ and $XY$. It is easy to verify that the angles of triangles $COX$ and $DOY$ are pairwise the same. Moreover we get $CX = \pi - BX = \pi - AY = DY$ and thus the triangles $COX$ and $DOY$ are again identical.

By the triangle inequality we get the estimate
\[
\label{eq:sph-est}
XY  \leq CD + CX + DY.
\]
Moreover, by construction $\angle COX \leq \pi/2$ and thus, using the spherical law of sines 
\[
\frac{\sin (\angle COX)}{\sin CX} = \frac{\sin(\angle OCX)}{\sin XY/2},
\]
we conclude that $CX = DY \leq XY/2$. By the triangle inequality we conclude that $CD \leq 2(CX + XO) \leq 2XY$, which gives the upper bound
\begin{equation}
\label{eq:sph-est2}
CD + CX + DY \leq 3XY.
\end{equation}

\emph{Step 3: dynamic problem.}
Finally, we introduce the dynamic version of the triangle inequality \eqref{eq:sph-est}. Recall that $\gamma_X(t) = \exp_X tv_x$, then the geodesic $\gamma_C(t) = \exp_C t(\Pi_{CX} v_x)$ satisfies $\gamma_C(t) = \gamma_X(t - \delta)$ for some $\delta \in \bbR$, implying that $\dist(\gamma_X(t), \gamma_C(t)) = \dist(\gamma_X(0), \gamma_C(0)) = CX$.

Analogously, let $\gamma_D(t) = \exp_D t(\Pi_{DY} v_y)$. Since $\|v_x\| = \|v_y\|$, by construction $\dist(\gamma_C(t), A) = \dist(\gamma_D(t), A)$ for all $t\in\bbR$. Thus, we conclude that points $C$ and $D$ run synchronously along corresponding geodesics, which implies that $\dist(\gamma_C(t), \gamma_D(t)) \leq \dist(\gamma_C(0), \gamma_D(0))) = CD$ since at $t=0$ the geodesic $CD$ is orthogonal to both geodesics $AC$ and $AD$.

Combining the above estimates and using inequality \eqref{eq:sph-est2} we obtain the dynamic version of the triangle inequality \eqref{eq:sph-est}, namely
\[
\begin{aligned}
    \MoveEqLeft\dist(\gamma_X(t), \gamma_Y(t)) \leq \dist(\gamma_X(t), \gamma_C(t)) + \dist(\gamma_Y(t), \gamma_D(t)) + \dist(\gamma_C(t), \gamma_D(t)) \\
    &\leq CD + CX + DY \leq  3\dist(\gamma_X(0), \gamma_Y(0)),
    \end{aligned}
\]
which concludes the proof.
\end{proof}

\end{document}